\documentclass{amsart}
\usepackage{amssymb,a4wide}
\usepackage{amsthm}
\usepackage{graphicx}
\usepackage[T2A]{fontenc}
\usepackage[latin1]{inputenc}
\usepackage{mathdots}
\usepackage{multirow}

\newcommand{\bg}{{\bar g}}
\newcommand{\Lie}{{\mathcal L}}

\newcommand{\ti}{{\tilde\imath}}
\newcommand{\tj}{{\tilde\jmath}}
\newcommand{\tk}{{\tilde k}}
\newcommand{\ta}{{\tilde \alpha}}
\newcommand{\R}{\mathbb R}
\newcommand{\wt}[1]{\widehat{#1}}

\newcommand{\tg}{\hat g}
\newcommand{\dfdx}[2]{\frac{\partial #1}{\partial #2}}
\newcommand{\Hol}[1]{Hol_p(#1)}

\newcommand{\comment}[1]{{}}
\newcommand{\wtSol}{\widetilde{Sol}}
\newcommand{\tr}{\operatorname{Tr}}
\newcommand{\const}{\operatorname{const}}

\newcommand{\Id}{\operatorname{Id}}
\renewcommand{\Im}{\operatorname{Im}}
\newcommand{\weg}[1]{}

\theoremstyle{plain}
\newtheorem{thm}{Theorem}
\newtheorem{lem}{Lemma}

\newtheorem{cor}{Corollary}

\theoremstyle{definition}

\theoremstyle{remark}
\newtheorem{rem}{Remark}
\newtheorem{ex}{Example}

\title[Degree of mobility for metrics of lorentzian  signature]{Degree of mobility for metrics of lorentzian  signature and parallel  (0,2)-tensor fields on cone 
manifolds}
\author{Aleksandra  Fedorova and Vladimir S. Matveev} 
\address{Mathematisches Institut, Friedrich-Schiller Universit\"at Jena\\
07737 Jena, Germany}  
\email{vladimir.matveev@uni-jena.de}
\begin{document}

  \begin{abstract}
  Degree of mobility of a (pseudo-Riemannian) metric is the dimension
  of the space of metrics geodesically equivalent to it. We describe
  all possible values of the degree of mobility on a simply connected
  n-dimensional manifold of lorentz signature. As an application we
  calculate all possible differences between the dimension of the
  projective and the isometry groups. One of the main new  technical
  results in the proof is the description of all parallel  symmetric $(0,2)$-tensor fields on cone manifolds of signature
  $(n-1,2)$. 
 \end{abstract}
  \maketitle
  
\section{Introduction.} 
\subsection{Main definitions and results.}

Let $(M^n,g)$ be a connected Riemannian (= $g$ is positively definite)
or pseudo-Riemannian manifold of dimension $n\ge 3$. Within the whole paper we assume that all objects are $C^\infty$-smooth.  We say that a
metric $\bar g$ on $M^n$ is \emph{geodesically equivalent} to $g$, if
every geodesic of $g$ is a (reparametrized) geodesic of $\bar g$. We
say that they are \emph{affinely equivalent}, if their Levi-Civita
connections coincide.

As we recall in Section~\ref{standard}, the set of metrics
geodesically equivalent to a given one (say, $g$) is in one-to-one
correspondence with nondegenerate solutions of the
equation~\eqref{basic}. Since the equation~\eqref{basic} is linear,
the space of its solutions is a linear vector space. Its dimension is
called the \emph{degree of mobility} of $g$ and will be denoted by
$D(g)$. Locally, the degree of mobility of $g$ coincides with the
dimension of the set (equipped by  its natural topology) of metrics
geodesically equivalent to $g$.

The degree of mobility is at least one (since $\const \cdot g$ is
always geodesically equivalent to $g$) and is at most $\tfrac{(n+1)(n+2)}{2}$,
which is the degree of mobility of simply-connected spaces of constant
sectional curvature. 

Our main result is the description of all possible values of the
degree of mobility on simply-connected manifolds of the lorentz
signature $(1,n-1)$:

\begin{thm}
  \label{th:main}
  \label{Bneq0}
  Let $(M^n, g)$, $n\geq 3$, be a connected simply-connected manifold
  of nonconstant curvature of riemannian or lorentzian  signature. Assume that there
  exists at least one metric which is geodesically equivalent to $g$,
  but is  not affinely equivalent to $g$. Then, the degree of mobility
  $D(g)$  is 
   equal to $\frac{k(k+1)}{2} + \ell$ for
  certain $0 \le k\le  n-2$ and $1 \le \ell \le \lfloor
  \tfrac{n-k+1}{3}\rfloor$.
\end{thm} 

In the theorem above the brackets ``$\lfloor \ , \ \rfloor$'' mean the
integer part.

\begin{thm}
  \label{thm:realization} 
  For any $n\geq 3$, $0\leq k\leq n-2$ and $\ell \in \{1,...,\tfrac{n-k+1}{3} \}$ such that $\tfrac{k(k+1)}{2}+\ell\geq 2$
  there exists a Lorentzian metric $g$
  on $\R^n$ such that it admits a metric $\bar g$ that is geodesically
  equivalent, but not affinely equivalent to $g$,  and such that  
  $D(g)=\frac{k(k+1)}{2}+\ell$.
\end{thm}

The condition $\frac{k(k+1)}{2}+\ell\geq 2$  in Theorem  \ref{thm:realization} is due to our assumption that 
there exists a metric $\bar g$ that is geodesically
  equivalent, but not affinely equivalent to $g$. Actually, a generic metric $g$  does not admit such a  metric, and in fact has $D(g)= 1$, see \cite{hall}.

\begin{figure} 
    \caption{Degree of mobility for low dimensions}
  \includegraphics[scale=0.7]{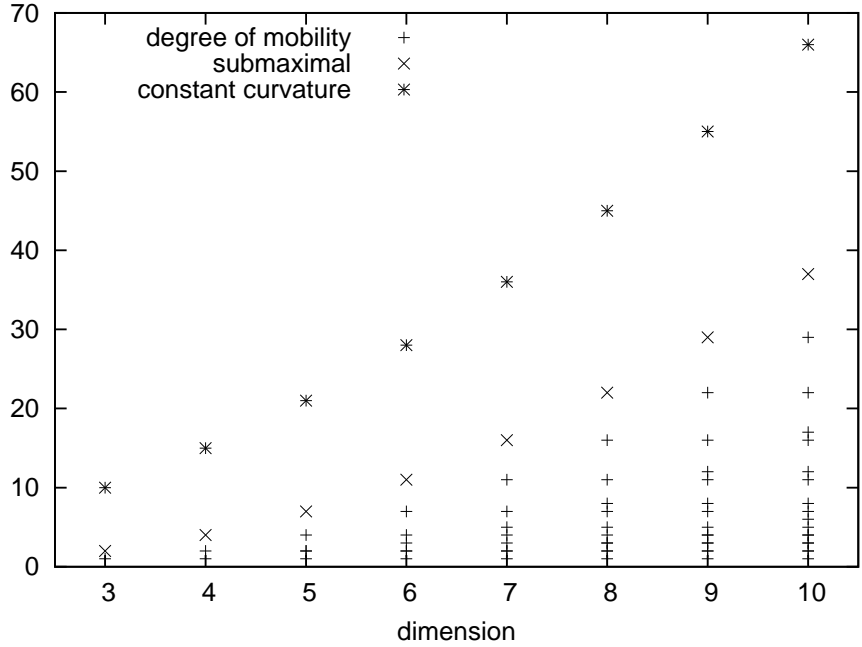} 
\end{figure}

The Riemannian version of Theorem~\ref{th:main} is known and is due to \cite{shandra,KMMS}: the principle idea is due to~\cite{shandra},  but the main result has a mistake which was corrected in~\cite{KMMS}.

We see that  the biggest degree of mobility of a metric of a nonconstant
curvature on a simply-connected manifold is $\tfrac{(n-1)(n-2)}{2} +
1= \tfrac{(n - 3) n }2+2$.  This value is known to be the
``submaximal'' value for metrics of all signatures, see~\cite[\S
1.2]{mikes} and \cite[Theorem 6.2]{dissertation_kiosak}. As we mentioned above, the  maximal value of the degree of mobility of a metric
on an $n$-dimensional manifold is achieved on simply-connected
manifolds of  constant sectional curvature and is equal to
$\tfrac{(n+2)(n+1)}{2}$.

Let us now comment on our assumptions in Theorem~\ref{th:main}. The
assumption that the manifold is simply-connected is important: the
degree of mobility of an isometric quotient can 
be smaller than the degree of mobility of the initial manifold.  For
example, for certain isometric quotients of the round 3-sphere, the
degree of mobility could be one,~\cite{topology}.   The assumption $n\ge 3$ is also important: the
dimension $n=2$ was studied already by Darboux~\cite{Darboux} and
Koenigs~\cite{konigs}, see also \cite{duna,kruglikov}. They have shown that in dimension $n=2$ the
degree of mobility of an arbitrary metric of nonconstant curvature  is 1,2,3,4. We see that the
list of possible degrees of mobility in dimension 2 is very different
from the list obtained by   the formula $D(g)= \frac{k(k+1)}{2}
+ \ell$ from Theorem~\ref{th:main}.

The assumption that there exists a metric that is geodesically equivalent to $g$ but not affinely equivalent to $g$ is also important since one can  construct examples of metrics (of arbitrary signature) 
on $\mathbb{R}^n$  with the degree of mobility equal to $\tfrac{(n-4)(n-3)}{2} + 2$, and for  $n\ge  5 $ 
this number is not in the list  of degrees of mobility given by Theorem \ref{th:main}. Of course, in the lorentzian signature, all metrics geodesically equivalent to the metrics from  these examples are affinely equivalent to them.

Unfortunately, we do not know whether the assumption that the metric
has riemannian or lorentzian signature is important. In dimension $n=3$,  all metrics
have, up to multiplication by $-1$, the riemannian or lorentzian
signature. In dimension $n=4$ one can show that the statement of our
theorem still holds (was essentially done in \cite{dissertation_kiosak}). 
Our proof does not work for metrics of other signatures though: we construct examples showing that one of the main tools of the proof, Theorem \ref{cov}, is wrong if the initial metric has other signatures.  

\subsection{Application to the dimension of the projective algebra.} 

A vector field whose (local) flow takes unparameterized geodesics to
geodesics is called a \emph{projective vector field}. Projective vector
fields satisfy the equation~\eqref{apl} in Section~\ref{standard}.
Since the equation~\eqref{apl} is linear, the space of its solutions
is a linear vector space, we will denote it by  $proj(g)$. Since every
killing vector field (i.e., a vector field whose flow acts by (local)
isometries) is evidently a projective vector field, the set of the
Killing vector fields which we denote by $iso(g)$ forms a vector
subspace of $proj(g)$.

\begin{thm}
  \label{th:apl}
  Let $(M^n, g)$ be a connected simply-connected $n\ge 3$-dimensional manifold of
  nonconstant curvature of riemannian or lorentzian  signature. Assume that $D(g)\geq
  3$ and that there exists at least one metric which is geodesically
  equivalent to $g$, but not affinely equivalent to $g$. Then,
  $$\dim proj(g)-
  \dim iso(g) = \frac{k(k+1)}{2} + \ell-1$$ for certain $k\in
  \{0,1,...,n-2\}$ and $1 \le \ell \le \lfloor
  \tfrac{n-k+1}{3}\rfloor$.
\end{thm} 

In the case of $D(g)=2$, we prove that 
the codimension of the space of homothety vector fields  (i.e., such that $\mathcal{L}_v g = \const g$) in the space
of projective fields is at most 1, see Lemma \ref{thm:proj/hom} is Section \ref{cod}. 

Note that the number $\dim proj(g)-
  \dim iso(g)$ is a natural number in the projective geometry. Indeed, since Knebelman \cite{knebelman,beltrami_short} it is known that if $\bar g$ is geodesically equivalent to $g$ then $\dim iso(g)= \dim iso(\bar g)$. Since evidently  $\dim proj(g)= \dim proj(\bar g)$, we have    $\dim proj(g)-
  \dim iso(g)= \dim proj(\bar g)-
  \dim iso(\bar g)$.

Note that there is almost no hope to obtain the possible dimensions of
$iso(g)$ (for manifolds of all dimensions), since the possible values of $\dim
iso(g)$ for homogeneous manifolds give too many combinatorical 
possibilities.  For every fixed dimension, it can be in principle done
though.  Moreover, as examples show, the lists of possible dimensions
of $iso(g)$ on a simply-connected $n$-dimensional manifold depend on
the signature of $g$ and, for certain $n$, are different for
Riemannian and Lorentzian metrics.

\subsection{Overview of known global results.} 

If the manifold $(M,g)$ is closed or the metrics are  complete,  the  natural analog of Theorem  \ref{th:main}   was known before and  is true for metrics of all signature: by \cite[Theorem 1]{KM}, if two complete metrics $g$ and $\bar g$  of nonconstant curvature on a $n\ge 3$-dimensional manifold  are geodesically equivalent but not affinely equivalent,   then $D(g)= 2$. 
By    \cite[Corollary 5.2]{Mounoud2010}, if two  metrics $g$ and $\bar g$  of nonconstant curvature on a closed  $n\ge 3$-dimensional manifold  are geodesically equivalent but not affinely equivalent,   then $D(g)=2$. 
If we merely  assume that 
 the  metric  $g$  is complete,  then, in the riemannian and in the 
 lorentzian case, the list of the degrees of mobilities  (on connected simply-connected manifolds) coincides with that of in Theorem \ref{th:main}.

\subsection{ Relation to parallel 
  symmetric $(0,2)$-tensor fields and  difficulties of the  lorentzian signature.} \label{plan}

The \emph{cone manifold} over $(M,g)$ is the manifold $\wt
M=\mathbb{R}_{>0}\times M$ endowed with the metric $\hat g$ defined by
$\hat g=\,dr^2+ r^2 g$ (i.e., in the local coordinate system
$(r,x^1,...,x^n) $ on $\wt M$, where $r$ is the standard coordinate
on $\R_{>0}$, and $(x^1,...,x^n)$ is a local coordinate system on $M$,
the scalar product in $\hat g$ of the vectors $u= u^0 {\partial_r} +
\sum_{i=1}^n u^i {\partial_{ x^i} } $ and $v= v^0 {\partial_r} +
\sum_{i=1}^n v^i {\partial_{x^i} } $ is given by $\hat g(u,v)= u^0
v^0+ r^2\sum_{i,j=1}^n g_{ij} u^i v^j$).
 
The degree of  mobility of metrics on $n$-dimensional manifold
appears to be closely related to the dimension of the space of
parallel  symmetric $(0,2)$-tensor fields of $n+1$-dimensional
cone manifolds. We will explain what we mean by ``closely related'' in
Section~\ref{standard10}.  For Riemannian metrics, this observation was
essentially known to Solodovnikov~\cite{s1} and was used
by Shandra in~\cite{shandra}, where, as we mentioned above, the Riemannian
version of our main Theorem~\ref{th:main} was essentially proved. In~\cite{KM},
the result of Solodovnikov was extended for all signatures, which
allows us to use it in our problem.  The assumption that the metric
$g$ has lorentzian signature $(1,n-1)$ implies that the (metric of the)
cone manifold which is used in the proof of Theorem~\ref{th:main} has
signature $(1,n)$ or $(n-1, 2)$.
  
In view of this relation between geodesically equivalent metrics and
parallel symmetric $(0,2)$-tensor fields, the following
statement is closely related to Theorem~\ref{th:main}:
\begin{thm}
  \label{cone} 
  Let $(\wt M^{n+1}, \hat g)$ be a connected simply-connected nonflat
  cone manifold of signature $(0,n+1)$, $(1,n)$ or $(n-1, 2)$.  Then,
  the dimension of the space of parallel symmetric
  $(0,2)$-tensor fields is $\frac{k(k+1)}{2} + \ell$, where $k $ is the
  dimension of the space of parallel vector fields, and $1 \le \ell
  \le \lfloor \tfrac{n-k+1}{3}\rfloor$.
\end{thm} 
 
Though Theorem \ref{cone} provides one  of the main steps  in the proof of Theorem \ref{th:main}, 
it is  not equivalent to Theorem \ref{th:main}. 
Actually, Theorem~\ref{cone} follows from Theorem~\ref{th:main}, modulo certain results of \cite{KM,Mounoud2010} we recall in Section \ref{sec:Bneq0}. If the
signature of $g$ is riemannian, Theorem~\ref{th:main} follows from
Theorem~\ref{cone}, though this implication needs additional work which will be  essentially done in Section \ref{sec:mu_neq_0}.

If $g$ has lorentzian  signature, proof of  Theorem~\ref{th:main} splits
into two parts: the first ``generic'' part is based on  
Theorem~\ref{cone} and the second part (``special case'') is   Lemma~\ref{thm:B=0_isotropic}.
 
Let us now explain two main steps in the proof of Theorem~\ref{cone},
which are Theorems~\ref{cov},~\ref{cov1} below. Besides providing an important step in the proof of Theorem \ref{th:main}, Theorem~\ref{cov}  could be interesting on its   own since investigation of parallel  tensor fields  on cone manifolds  is a classical topic, see  for example \cite{Cortes2009,Ga, Mounoud}. 

 Fix a point $p\in \wt
M$ (where $(\wt M, \hat g)$ is a simply-connected cone
manifold). Consider the holonomy group $\Hol{\hat g} \subset SO(T_p\wt
M, \hat g_{p})$ of the metric $\hat g$.  Consider the decomposition of
$T_p\wt M$ in the direct product of  mutually orthogonal $\hat
g$-nondegenerate subspaces invariant w.r.t. the action of the holonomy
group
\begin{equation}
  \label{decomposition} 
  T_p\wt M = V_0 \oplus V_1\oplus ... \oplus V_\ell.  
\end{equation} 
We assume that $V_0$ is flat in the sense that the holonomy group acts
trivially on $V_0$, and that the decomposition is maximal in the sense
that for all $i=1,...,\ell$ it is not possible to decompose $V_i$ into
the direct product of two nontrivial (i.e., of dimension $\ge 2$)
$\hat g$-nondegenerate subspaces invariant w.r.t. the action of the
holonomy group.  We allow $\dim(V_0)= 0$ but assume $\dim(V_i)\ge 2$
for $i\ne 0$.
  
It is well known that parallel  symmetric
$(0,2)$-tensor fields on a simply-connected manifold are in the one-to-one
correspondence with  bilinear symmetric $(0,2)$-forms on $T_p\wt M$ invariant
with respect to $\Hol{\hat g}$.

We denote by $g_i$, $i= 0,\dots ,\ell$,  the restriction of the metric $g$
to the subspace $V_i$, considered as a $(0,2)$-tensor on $T_pM$: for
the vectors $v = v_0 + v_1+...+v_\ell$ and $u = u_0 + u_1+...+u_\ell$
of $T_p\wt M$ (where $v_\alpha, u_\alpha\in V_\alpha$) we put
$$g_i(v_0 + v_1+...+v_\ell, u_0 + u_1+...+u_\ell) = g(v_i, u_i).$$
$g_i$ is evidently invariant w.r.t.  $\Hol{\hat g}$.

\begin{thm}
  \label{cov}
  Let $(\wt M,\hat g)$ be a simply-connected cone manifold of dimension $n+1$. Assume
  $\hat g$ has signature $(1, n)$ , $(n-1, 2)$, or the riemannian
  signature $(0,n+1)$,  and consider the (maximal)
  decomposition~\eqref{decomposition}. We denote by
  $\{\tau_1,...,\tau_k\}$ a  basis in the space of 1-forms on $T_p
  \wt M$ that are invariant with respect to  the
  holonomy group.
  
  Let $A$ be a symmetric bilinear form  on $T_p \wt M$ such that  it is  invariant with
  respect to the holonomy group. Then, there exists   a
  symmetric $k\times k$-matrix $c_{ij}\in \mathbb{R}^{k^2}$ and constants
  $C_1,...,C_\ell\in \mathbb{R}$ such that
  \begin{equation}
    \label{1}
    A= \sum_{i,j=1}^k c_{ij} \tau_i\tau_j+ \sum_{i=1}^\ell C_i g_i .
  \end{equation} 
\end{thm} 

Evidently, any bilinear form  $A$ given by the formula~\eqref{1} is invariant
with respect to the  holonomy group and is symmetric.

In the case when the metric $\hat g$ is Riemannian, Theorem~\ref{cov}
is well-known and is essentially due to de Rham~\cite{DR}. Moreover,
in this case it is true without the assumption that $(\wt M, \hat g)$
is a cone manifold. The classical way to formulate Theorem~\ref{cov}
in the Riemannian setup is as follows: there exists a coordinate
system
$$x =(\underbrace{x_0^1,...,x_0^{k_0}}_{\bar x_0}, \underbrace{x_1^1,...,x_1^{k_1}}_{\bar x_1}, ... ,
\underbrace{x_\ell^1,...,x_\ell^{k_\ell}}_{\bar x_\ell})$$ in a neighborhood of $p$ such that in
this coordinate system the metric has the block-diagonal form (with
the blocks of dimensions  $k_0\times k_0$, ... , $k_\ell\times k_\ell$)
$$
\hat g =
\begin{pmatrix}
  g_0 & &  \\
  & g_1& &  \\
  & & \ddots &  \\
  & & & g_\ell
\end{pmatrix}
$$
such that the entries of each matrix $g_i$ depend on the coordinates $
x_i^1,...,x_i^{k_i}$ only, such that the metric $g_0$ is the flat
metric $(dx_0^1)^2 + ... + (dx^{k_0}_0)^2$, and such that the holonomy
group of each metric $g_i$ for $i\ne 0$ is irreducible. For this
metric, every symmetric bilinear form  invariant with respect to  the holonomy group is given by
\begin{equation}
  \label{1b}
  A= \sum_{i,j=1}^{k_0} c_{ij} dx_0^i\, dx_0^j + \sum_{i=1}^\ell C_i g_i,
\end{equation}
where $c_{ij}$ is a symmetric $k_0\times k_0$-matrix. The relation
between the formulas~\eqref{1} and~\eqref{1b} is as follows: in the
formula~\eqref{1b}, the 1-forms invariant with respect to $\Hol{\hat
  g}$ are (essentially) the one-forms on $V_0$ (and therefore $k=
k_0=dim(V_0)$ and as the basis in the space of 1-forms invariant w.r.t. $\Hol{\hat
  g}$  we can take $dx_0^1,...,dx_0^{k_0}$). In the other signatures, there may exist invariant
one-forms on $T_p\wt M$ that are not 1-forms on $V_0$.

In the case when the metric $\hat g$ has lorentzian signature,
Theorem~\ref{cov} is also known~(see for example \cite{SK}) and is also true
without the assumption that $(\wt M, \hat g)$ is a cone manifold. The
new part of Theorem~\ref{cov} is when the signature is $(n-1, 2)$, as
example \ref{ex1} in Section \ref{counter} shows, in this case the assumption that the metric is a cone
metric is essential.

Moreover, the assumption that the signature of the metric $g$ is riemannian, lorentzian, or  $(n-1, 2)$  is important for Theorem \ref{cov}, see (counter)example \ref{ex2} in Section \ref{counter}. 
 
Theorem~\ref{cov} describes all parallel  symmetric
$(0,2)$-tensor fields on cone manifolds of signatures $(0,n+1)$, $(1,n)$ and $(n-1,
2)$. The next theorem counts the dimensions of the space of such
tensor fields.
  
\begin{thm}
  \label{cov1}
  Under the assumptions of Theorem~\ref{cov}, the number $k$ is at
  most $n-2$ and the number $\ell$ is at least $1$ and at most $\lfloor
  \tfrac{n-k+1}{3}\rfloor$.
\end{thm}

Combining Theorems~\ref{cov} and~\ref{cov1}, we obtain
Theorem~\ref{cone}. Now, as we explained above, Theorem~\ref{cone} is
essentially equivalent to the ``generic'' part of the
Theorem~\ref{th:main}; and the ideas used in the proof of Theorem \ref{cov}, \ref{cov1}  will  also be seen  in  the ``special'' part of the proof  of Theorem~\ref{th:main}.

\section{ Degree of mobility as the dimension of the  space of parallel symmetric $(0,2)$ tensor fields on the cone. }  \label{standard10}
\subsection{Geodesically equivalent metrics, Sinjukov equation, and degree of   mobility.}
\label{standard}

Whenever the tensor index notation are used, we consider $g$ as the
background metric (to low and rise indexes), sum with respect to
repeating indexes, and denote by comma ``,'' the covariant
differentiation w.r.t. the Levi-Civita connection of $g$. The
dimension of our manifold will be denoted by $n$; we assume $n\geq 3$.
 
As it was known already to Levi-Civita~\cite{Levi-Civita}, two
connections $\nabla= \Gamma_{jk}^i $ and $\bar \nabla= \bar
\Gamma_{jk}^i $ have the same unparameterized geodesics, if and only
if their difference is a pure trace: there exists a 1-form $\phi $
such that
\begin{equation}
  \label{c1}
  \bar \Gamma_{jk}^i - \Gamma_{jk}^i = \delta_k^i\phi_{j} +
  \delta_j^i\phi_{k}.
\end{equation} 
  
If $\nabla$ and $\bar \nabla$ related by~\eqref{c1} are Levi-Civita
connections of metrics $g$ and $\bar g$, then one can find explicitly
(following Levi-Civita~\cite{Levi-Civita}) a function $\phi$ on the
manifold such that its differential $\phi_{,i}$ coincides with the 1-form  $\phi_i$: indeed, contracting~\eqref{c1} with respect to $i$
and $j$, we obtain $\bar \Gamma_{p i}^p = \Gamma_{p i}^p + (n+1)
\phi_{i}$.  On the other hand, for the Levi-Civita connection $\nabla$
of a metric $g$ we have $ \Gamma_{p k}^p = \tfrac{1}{2} \frac{\partial
  \log(|det(g)|)}{\partial x_k} $.  Thus,
\begin{equation}
  \label{c1,5}
  \phi_{i}= \frac{1}{2(n+1)}  \tfrac{\partial }{\partial x_i}
  \log\left(\left|\tfrac{\det(\bar g)}{\det( g)}\right|\right)=
  \phi_{,i}
\end{equation} 
for the function $\phi:M\to \mathbb{R}$ given by
\begin{equation}
  \label{phi}
  \phi:= \tfrac{1}{2(n+1)} \log\left(\left|\tfrac{\det(\bar g)}{\det(
        g)}\right|\right).
\end{equation}
In particular, the derivative of $\phi_i$ is symmetric, i.e.,
$\phi_{i,j}= \phi_{j,i}$.
 
The formula~\eqref{c1} implies that two metrics $g$ and $\bar g$ are
geodesically equivalent if and only if for a certain $\phi_{i}$ (which
is, as we explained above, the differential of $\phi$ given
by~\eqref{phi}) we have
\begin{equation}
  \label{LC}
  \bar g_{ij, k} -2 \bar g_{ij} \phi_{k}-\bar g_{ik}\phi_{j}- \bar
  g_{jk}\phi_{i}= 0, 
\end{equation} 
where ``comma'' denotes the covariant derivative with respect to the
connection $\nabla$.  Indeed, the left-hand side of this equation is
the covariant derivative with respect to $\bar \nabla$, and vanishes
if and only if $\bar \nabla$ is the Levi-Civita connection for $\bar
g$.

The equations~\eqref{LC} can be linearized by a clever substitution:
consider $a_{ij}$ and $\lambda_i$ given by
\begin{eqnarray}
  \label{a}
  a_{ij} &=   &e^{2\phi} \bar g^{p q} g_{p i} g_{q j},\\
  \label{lambda} 
  \lambda_{i} & = &  -e^{2\phi}\phi_p \bar g^{p q} g_{q i},
\end{eqnarray}
where $\bar {g}^{p q}$ is the tensor dual to $\bar g_{p q}$: \ $\bar
{g}^{p i}\bar g_{p j}= \delta_j^i$.  It is an easy exercise to show
that the following linear equations for the symmetric $(0,2)$-tensor
$a_{ij}$ and $(0,1)$-tensor $\lambda_i$ are equivalent to~\eqref{LC}:
\begin{equation}
  \label{basic} 
  a_{ij,k}= \lambda_i g_{jk} + \lambda_j  g_{ik}. 
\end{equation} 
One may consider the equation~\eqref{basic} as a linear PDE-system on
the unknown $(a_{ij},\lambda_k)$; the coefficients in this system
depend on the metric $g$.

One can also consider~\eqref{basic} as a linear PDE-system on the
components of the tensor $a_{ij}$ only, since the components of
$\lambda_i$ can be obtained from the components of $\nabla_k a_{ij}=
a_{ij,k}$ by linear algebraic manipulations.  Indeed,
multiplying~\eqref{basic} by $g^{ij}$ we obtain 
\begin{gather}
  \label{eq:lambda}
  \lambda_{k} =
  \tfrac{1}{2}(a_{ij}g^{ij})_{,k}=\tfrac{1}{2}\left(\tr_g(a)\right)_{,k}
\end{gather}

Since~\eqref{basic} is a system of linear PDE, the set of its
solutions is a linear vector space. Its dimension will be called the
\emph{degree of mobility} of $g$ and denoted by $D(g)$. Clearly, $D(g)
\ge 1$, since $a_{ij} =g_{ij}$ is a solution of~\eqref{basic}. It is
known (see for example~\cite[p.134]{Sinjukov}) that $D(g)\le
\tfrac{(n+1)(n+2)}{2}$.

\subsection{Metrics with $D(g)\ge 3$, extended system,  and plan of the proof of Theorem \ref{th:main}.} \label{standard1}

\begin{thm}[\cite{KM}]
  \label{thm:mg+Ba}
  Let $(M^n,g)$, $n\ge 3$,  be  a connected pseudo-Riemannian manifold such that  $D(g) \geq
  3$. Then there exists a constant $B$ such that for every  solution
  $(a_{ij},\lambda_i)$ of~\eqref{basic} there exists a smooth  function
  $\mu$ such that the following system
  \begin{gather}
    \label{eq:mg+Ba}
    \left\{
      \begin{array}{ll}
        a_{i j,k}&=\lambda_i g_{j k} + \lambda_j g_{i k}\\
        \lambda_{i,j}&=\mu g_{i j}+B a_{i j}\\
        \mu_{,i}&=2B\lambda_i
      \end{array}
    \right. 
  \end{gather}
  is satisfied.
\end{thm}

Thus, the degree of mobility of the metric $g$ is equal to the
dimension of the space of solutions $(a,\lambda,\mu)$ of the
``extended system''~\eqref{eq:mg+Ba}.

Note that the constant $B$ is a metric invariant of $g$ (in the sense that a metric can not have two nontrivial solutions with different $B$, see \cite[\S 2.3.5]{KM})  but
it is not a projective invariant: for a metric $\bar g$ that is
geodesically equivalent to $g$ we may have $\bar B:= B(\bar g) \neq B$ (see
Section~\ref{sec:mu_neq_0}).

In Section~\ref{sec:Bneq0} we reduce the case $B\neq 0$ to   $B=-1$ and
then show that  the solutions of~\eqref{eq:mg+Ba} and
parallel symmetric $(0,2)$-tensor fields on the cone manifold are in one-to-one correspondence. 
In this setting, Theorem \ref{th:main} follows from  Theorems \ref{cov}, \ref{cov1}. 
 
In the case $B=0$  we consider two  subcases.  In
Section~\ref{sec:zeroB},  we assume  that 
  the extended
system~\eqref{eq:mg+Ba} admits a solution $(a,\lambda,\mu)$ with
  $\mu\not\equiv  0$.  In this subcase we    can  (locally) 
  find a  metric $\bar g$ that is  geodesically equivalent to $g$, has the same signature as $g$ and 
such that $\bar B= B(\bar g)<0$. 
   Since evidently $D(\bar g)=D(g)$ and the case
$B\neq 0$ has been already solved, we are done (though some additional
work is required  to make a transition  ``local'' $\rightarrow$ ``on a simply connected manifold'', see Section~\ref{sec:noncompact} for details).
 
Finally, in Section~\ref{sec:B=0_isotropic} we consider the ``special''
case,  when the extended system admits only solutions $(a,\lambda,\mu)$
with $\mu=0$. In this case all metrics $\bar g$ geodesically
equivalent to $g$ have  $\bar B=0$. In Section \ref{sec:B=0_isotropic}, we study and describe  such metrics,  
 calculate their  degrees  of mobility, and show that  they are  still in  the list
from Theorem~\ref{th:main}. 

\subsection{Metric cone and its Levi-Civita connection.}
\label{sec:cone}

By the  
\emph{metric cone} over  $(M,g)$ we understand  the  product manifold $\wt
M=\R_{{}>0}(r)\times M(x)$ equipped by  the  metric $\tg$ such that  in the
coordinates $(r,x)$ its matrix has the form
\begin{gather}
  \label{eq:conemetric}
  \tg(r,x)=\left(
    \begin{array}{cc}
      1& 0 \\
      0& r^2 g(x)
    \end{array}
  \right).
\end{gather}
The coordinates such that a metric has the form \eqref{eq:conemetric} will be called the {\it cone coordinates}.

For further use we calculate the Levi-Civita connection on $\wt M$ in the cone coordinates.

\begin{lem}[Folklore, see for example~\cite{Cortes2009,Mounoud2010}]
  \label{thm:coneconnection}
  Levi-Civita connection $\wt\nabla=\{\wt\Gamma^{\ti}_{\tj \tk}\}$
  corresponding to metric $\tg$ on $\wt M$ is given  by formula:
  \begin{gather}
    \wt\Gamma^0_{0 0}=0, \ 
    \wt\Gamma^i_{0 0}=0,\\
    \wt\Gamma^0_{j 0}=\wt\Gamma^0_{0 k}=0,\\
    \wt\Gamma^i_{j 0}=\frac{1}{r}\,\delta^i_j,\qquad \wt\Gamma^i_{0
      k}=\frac{1}{r}\,\delta^i_k,\\
    \wt\Gamma^0_{j k}={}-r\cdot g_{j k}(x),\\
    \wt\Gamma^i_{j k}=\Gamma^i_{j k}(x),
  \end{gather}
  where $\Gamma^i_{j k}$ are Christoffel symbols of the Levi-Chivita
  connection determined by metric $g$ on $M$, and indicies $i,j,k$
  take values from $1$ to $n$.
\end{lem}
\begin{proof}
  It is an easy exercise: we substitute the components of $\tg$ in the
  formula:
  $$\wt \Gamma^{\ti}_{\tj \tk}=\frac{1}{2}\, \tg^{\ti\ta}
  \left({}-\dfdx{\tg_{\tj\tk}}{x^{\ta}} + \dfdx{\tg_{\ta\tk}}{x^{\tj}}
    + \dfdx{\tg_{\tj\ta}}{x^{\tk}}\right)$$ and put $\ti,\tj,\tk$ to
  be equal to 0 or to certain $i,j,k$ respectively.
\end{proof}

\subsection{Cone structure as the  existence of a  positive  solution of  \eqref{eq:hom}.}
\label{sec:hom}

\begin{lem}
  \label{thm:hom}
  Pseudo-Riemannian  $n+1$-dimensional manifold $(\wt M,\tg)$ is locally isometric to a cone manifold   if and only if,  
   for any $P\in \wt M$ there exists  
   a positive function $v$ on $U(P)$  such that
  \begin{gather}
    \label{eq:hom}
    \left\{
      \begin{array}{c}
        v_{,ij}=\tg_{ij}, \\
        v_{,i}\,v_{,}\vphantom{v}^{i}= 2 v.
      \end{array}
    \right.
  \end{gather}
\end{lem}
\begin{proof}
  $\Rightarrow$ Let $(\wt M,\tg)$ be locally the cone over 
  $(M,g)$. Then there exist coordinates $(r,x)$, such that $\tg$ has
  the form~\eqref{eq:conemetric}. By direct calculations we see that 
  the function $v=\frac{1}{2} r^2$  satisfies \eqref{eq:hom}.

  $\Leftarrow$ Suppose  $v$ is  a positive function in $U(P)$ satisfying~\eqref{eq:hom}. We consider 
   $r=\sqrt{2v}$ and its gradient $r_{,}^{ \ i}$.   By direct calculation, we see $$r_{,i}\,
  r_,\vphantom{r}^{i}=\frac{v_{,i}}{\sqrt{2v}}\cdot
  \frac{v_,\vphantom{v}^{i}}{\sqrt{2v}}=\frac{1}{2v}\cdot 2 v=1. $$
  This  in particular implies that  the differential of $r$ nowhere vanishes.

  Consider the $n$-dimensional hypersurface $S$ defined by the equation
  $r=r(P)$. Let $(x_1,\dots,  x_n)$ be a local coordinate system  on $S$.

  \par Let us now use $r$ and the coordinates $(x_1,...,x_n)$ to construct a  coordinate system  in a neighborhood $P$. More precisely,
  for every point $Q=(x_1,\dots,x_n)\in S$ there exists the unique curve
  $\gamma_Q:(r(P)-\varepsilon, r(P) + \varepsilon)\to \wt M  $ such that  $$\dot{\gamma_Q}^i(t)=r_{,}^{\ i}\mbox{ and }
  \gamma_Q(r(P))=Q.$$ 
 Clearly, the mapping $(t, Q)\mapsto \gamma_Q(t)$ is a local diffeomorphism  and therefore defines a coordinate system $(t,x_1,...,x_n)$ in a neighborhood of $P$. 
   Because $r_{,}^{\ i}$ is  the gradient of $r$,  the value of the function $r$ at the 
  points $\gamma_Q(t)$ is equal to $r$, so this coordinate system actually reads $(r,x_1,\dots,x_n)$.

  Let us now  show that in these coordinates the metric $\tg$ has
  the form~\eqref{eq:conemetric}.  Using~\eqref{eq:hom},  we calculate
  $$r_{,ij} = \wt \nabla_j \left(\frac{v_{,i}}{r}\right)=
  \frac{1}{r}(\tg_{ij} - r_{,i} r_{,j}).$$

  By the construction of the  coordinates $(r,x)$, we have  
  \begin{equation}\label{-2}\partial_r \tg_{ij}=\Lie_{\partial_r}(g_{ij})=2 r_{,ij}
  =\frac{2}{r}(\tg_{ij} - r_{,i} r_{,j})\end{equation}
  For $i,j\neq 0$  the  equation \eqref{-2} reads  $\partial_r
  \tg_{ij}=\frac{2}{r}\tg_{ij}$. This equation could be viewed as an ODE; solving it 
  we obtain
  $\tg_{ij}(r,x)= r^2 g_{ij}(x)$, where $g_{ij}(x)$ is the restriction of the metric $\hat g$ to $S$ written  in the coordinates $x_1,...,x_n$.    Since the  $r_,^{\ i}$ is the gradient of $r$ and therefore is    orthogonal to $\{(r,x_1,...,x_n)\mid r=\const\}$, we have  
  $\tg_{0j}=\tg_{i0}=0.$ 
  Now,  $\tg_{00}=r_{,i} r_,\vphantom{r}^{i}=1$. Combining all these, we see that  in the coordinates
  $(r,x_1,\dots,x_n)$ the metric $\tg$ is given by~\eqref{eq:conemetric}.
\end{proof}

\begin{rem} From the first equaiton of \eqref{eq:hom} we see that  a
 solution $v$ of \eqref{eq:hom} has nonzero differential at every point of a certain everywhere dense open subset of $M$. Then, $v$ is not zero at every point of   a certain everywhere dense open subset of $M$. By Lemma \ref{thm:hom},  near  the points  where $v$ is positive, $g$ is isometric  to a cone metric. Since, for a negative solution  $v$ for $g$ 
  the  (positive)  function  $-v$ is a solution of \eqref{eq:hom} for  $g'= -g$, the metric $-g$ is locally a cone metric. 
\end{rem}

\begin{rem} Actually, the first equation of \eqref{eq:hom} almost implies the second. Indeed, if $v$
satisfies the first equation of \eqref{eq:hom}, then the function 
$\tfrac{1}{2} v_,^{\ i} v_{,i} $ has differential $\left(\tfrac{1}{2} v_,^{\ i} v_{,i}\right)_{,k}= v_,^{\ i} g_{ ik } = v_{,k}$ implying that  for a certain constant $C$  the function $v + C$ satisfies \eqref{eq:hom}. Moreover, if a 1-form $v_i$ satisfies  $v_{i,j}= g_{ij}$, then it is closed so  there exists a function $v$ such that $v_{,i}=v_i$  provided the manifold is simply connected. 
\end{rem} 

\subsection{Properties of  the cone vector field.}
By Lemma~\ref{thm:hom}, cone manifolds  are (locally) 
characterized by the existence  of  a positive  function $v$  satisfying  \eqref{eq:hom}. 
Its  gradient  $\vec v:= v_,^{\ i}$ will be called a {\it cone vector field.} 

\begin{lem}
\label{thm:Rv0}
Let $(M,g)$ be a $2$-dimensional manifold, and let $v_i$  be a $1$-form such that   $v_{i,j}= g_{ij}$.  Then,  $M$ is flat.
\end{lem}
\begin{proof}
  By the second equation of \eqref{eq:hom},  $v_i\neq 0$ on an everywhere dense open subset of $M$.  
Take a point $P\in \wt  M$ from this subset and choose a basis in $T_p\wt M$ such   that  in this basis $\vec v:= v_,^{\ i}=\left(
    \begin{array}{cc}
      1\\
      0
    \end{array}
  \right)$. By the definition of the Riemannian curvature $R^i_{\ jk\ell}$,  we have 
  \begin{gather}
    \label{eq:cone-in-the-kernel}
    R^i_{\ 1 k\ell }=R^i_{\ jk\ell} v^j{=}\nabla_k\nabla_\ell v^j -
    \nabla_k\nabla_l v^j   \stackrel{\eqref{eq:hom} }=   \nabla_k \delta^j_\ell -
    \nabla_\ell \delta^j_k =  0 \quad \mbox{for all $i,k,l=1, 2.$}
  \end{gather}
  Then, by the symmetries of the Riemannian curvature tensor we see that the
  component $R_{ijk\ell}=0$, when $i,j,k$ or $\ell$ is equal to $1$. Since
  the only remaining component $R_{2222}$ is also zero, $R^i_{\
    jk\ell}\equiv 0$ and the metric $g$ is flat.
\end{proof}

\begin{lem}
  \label{thm:cone-not-parallel}
  Assume  $v$ satisfies  \eqref{eq:hom}.   Then, for any parallel vector field $u\ne 0$, for 
   every point $P\in M$ such that the  Riemannian curvature
  tensor $R^i_{\ jk\ell} $ is not zero,  the vectors  $\vec v:= v_,^{\ i}$ and $u^i$ 
   are not proportional at $P$.
\end{lem}
\begin{proof}
  As we explained  in the proof of Lemma \ref{thm:Rv0}, see~\eqref{eq:cone-in-the-kernel} there,  we have $v_{,i} R^i_{\ jk\ell}=0$.
  We  covarinatly differentiate this equation to obtain:
  $$v_{i,m} R^i_{\ jk\ell}+v_{,i} R^i_{\ jk\ell,m}=0\ \ \textrm{implying} \ \  R_{mjk\ell}+v_{,i} R^i_{\ jk\ell,m}=0.$$ 

  Since $R_{mjk\ell}\ne 0$ at $P$, we have  \begin{equation} \label{-1} v_{,i} R^i_{\ jk\ell,m}\neq 0.\end{equation}   But  by definition of $R^i_{\ jk\ell}$, for any parallel vector field $u$ we have  
   $u^j R^i_{\ jk\ell} =\nabla_k\nabla_\ell u^j -
    \nabla_k\nabla_{\ell} u^j =0$. Covariantly 
    differentiating the last equation and using the symmetries of the curvature tensor, we obtain  $u_i R^i_{\ jk\ell,m} = 0$.  Combining this with 
    \eqref{-1}, we see  that  the vectors $u$ and $\vec v$ are not proportional. 
\end{proof}

\subsection{Direct product and decomposition of cone manifolds.}

\begin{lem}
  \label{thm:split}
  Consider   the direct product $$(\wt M,\tg)=(M_1,\overset{1}{g})\times (M_2,\overset{2}g)$$
  and assume that a function $v$  on $\wt M $ satisfies \eqref{eq:hom}. 
  
  Then, for any $s=1,2$ there exists a  function $\overset{s}v$ von $M_s$  satisfying 
   \eqref{eq:hom} (w.r.t. the metric $\overset{s}g$ on $M_s$). The  differential 
    $\overset{s}{v}_{, i}=d \overset{s}{v}$  of the  function $\overset{s}{v}$ is not zero at almost every point.  Moreover,  $\overset{s}{v}_{, i} R^i_{\ jkm} =  0$, where $R$ is the curvature tensor of $\hat g$.    
\end{lem}

\begin{rem} 
Withing the whole paper we understand ``almost everywhere'' or ``at almost every point'' in the topological sense: a property is fulfilled almost everywhere or at almost every point if the set of the points where it is fulfilled is  open and everywhere dense. 
\end{rem} 

\begin{proof}
  Let us consider the decomposition $v_{,i} = \overset{1}{v}_i+\overset{2}{v}_i\, ,$
  where $\overset{s}{v}_i$ is the orthogonal projection of $v_i$ to
  $TM_s$.  Let us choose coordinates $x_1,...,x_n$ on $M$  such that
  $x_1,\dots, x_k$ are coordinates on $M_1$ and $x_{k+1},\dots, x_n$ are
  coordinates on $M_2$.  For $i,j\leq k$ we have 
  $$\overset{1}{\nabla}_j \overset{1}{v}_i = \wt \nabla_j v_i=g_{i
    j}=\overset{1}{g}_{i j}$$
   implying that the components of $\overset{1}{v}_i$ depends  on 
   the coordinates  $x_1,\dots, x_k$  only and could be viewed as a 1-form on $M_1$. 
   Next, consider  the function $\overset{1}v:=  \tfrac{1}{2} \overset{1}v_i  {\overset{1}v}{}^i$ which also depend on the coordinates  $x_1,\dots, x_k$  only and  can be viewed as a function on $M_1$. We have 
   $$\overset{1}{v}_{,i}= \tfrac{1}{2} \left(\overset{1}v_k  {\overset{1}v}{}^k\right)_{,i}= \overset{1}{g}_{ik}\overset{1}{v}{}^k= \overset{1}v_i.$$   
Thus,  $\overset{1}{v}$ satisfies \eqref{eq:hom}. Similarly we can  prove the existence of a function 
   $\overset{2}v$ on $M_2$  satisfying  \eqref{eq:hom}.   
   
   The first equation of \eqref{eq:hom} implies that the differentials of $\overset{s}v_i$ are nonzero at almost every point. 
  
  Since  the curvature tensor of $\hat g$ is the direct sum of the curvature  tensors of $\overset{1}{g}$ and  $\overset{2}{g}$, and since as   we explained  in the proof of Lemma \ref{thm:Rv0},  $\overset{s}{v}_{, i} $ satisfies $\overset{s}{v}_{, i} \overset{s}R{}^i_{\ jkm} =  0$, where $\overset{s}R{}^i_{\ jkm}$ is   the curvature tensor  of $\overset{s}{g}$, we have $\overset{s}{v}_{, i} R^i_{\ jkm} =  0$.
\end{proof}

\begin{rem} 
In Lemma \ref{thm:split} we do not require that the functions (which were denoted by $\overset{1}v$, $\overset{2}{v}$ in the  proof)  
on the manifolds  $M_1, M_2$ are positive. It is easy to construct an example such that $\overset{1}v$ is positive and  $\overset{2}{v}$ is negative. 
\end{rem}

\begin{lem}
  \label{thm:glue}
 Assume    $(\wt
  M,\tg)$ is the  direct product of two manifolds,  $$(\wt
  M,\tg)=(M_1,\overset{1}g)\times (M_2,\overset{2}g),  $$ where 
 every  $(M_i,\overset{i}g)$, $i=1,2$,  is a    cone manifold. Then,
  $(\wt M, \tg)$ admits a positive function satisfying \eqref{eq:hom}.
\end{lem}
\begin{proof}
  By Lemma~\ref{thm:hom},  there exists  positive functions
  $u(x)$ on $M_1$ and $v(y)$ on $M_2$ such that
  \begin{gather}
    \begin{array}{cc}
      u_{,ij}=\overset{1}g_{ij},& v_{,ij}=\overset{2}g_{ij}, \\
      u_{,i}u^{,i}= 2 u,& v_{,i}v_{,}^{\ i}= 2 v.
    \end{array}
  \end{gather}
  Then, the function $w(x,y)=u(x)+v(y)$ is a positive function on $\wt M$
  and satisfies~\eqref{eq:hom}. Indeed,
  $$w_{,i j}= \left(
    \begin{array}{cc}
      u_{,ij}& 0 \\
      0& v_{,ij}
    \end{array}
  \right)=\tg_{ij}$$
  and $w_{,i}w^{,i}=u_{,i}u_,^{\ i}+v_{,i}v_,^{\ i}=2u + 2v = 2w$. 
\end{proof}

\begin{rem} Both Lemmas above are true for the direct products of arbitrary number of manifolds: the proofs survive without any changes.
\end{rem}

\subsection{Solutions of the extended system with $B\neq 0$ as
  parallel  $(0,2)$-tensor fields on the cone.}
\label{sec:Bneq0}

Let $(M,g)$ be a connected  $n\ge 3$-dimensional pseudo-Riemannian manifold with $D(g)\geq
3$. We consider the extended system~\eqref{eq:mg+Ba} and assume
$B\neq 0$.

By  Theorem~\ref{thm:mg+Ba}, the degree of mobility of $g$ is equal to the dimension of the space of solutions of~\eqref{eq:mg+Ba}. Our goal is to construct an isomorphism between the space of the solutions  
  of ~\eqref{eq:mg+Ba} and the space of 
 parallel symmetric $(0,2)$-tensor fields on a  cone manifold.

First we renormalize the metric in order to obtain $B=-1$.

\begin{lem}
  \label{thm:two_cases} Let  $(a_{ij}, \lambda_i, \mu)$ satisfy ~\eqref{eq:mg+Ba} with $B\ne 0$.  
  Then, $(a'_{ij}:= -Ba,  ,\lambda'_i:= \lambda_i, \mu':= -\tfrac{1}{B} \mu)$ satisfies  ~\eqref{eq:mg+Ba}   for the  metric $g'= -\tfrac{1}{B}g$ and $B(g')= B'= -1$.
\end{lem}
\begin{proof} We substitute  $(a'_{ij}:= -Ba,  ,\lambda'_i:= \lambda_i, \mu':= -\tfrac{1}{B} \mu)$ and $g'= -\tfrac{1}{B}g$ in the system 
  \begin{gather}
    \label{eq:m'g'+B'a'}
    \left\{
      \begin{array}{ll}
        a'_{i j,k}&=\lambda'_i g'_{j k} + \lambda'_j g'_{i k}\\
        \lambda'_{i,j}&=\mu' g'_{i j}-a'_{i j}\\
        \mu'_{,i}&=-2\lambda'_i
      \end{array}
    \right.
  \end{gather}
 and see that it is fulfilled.
\end{proof}

Thus, if $B\ne 0$,  we can assume $B=-1$. In this setting 
the  system~\eqref{eq:mg+Ba}  reads
\begin{gather}
  \label{eq:mg-a}
  \left\{
    \begin{array}{ll}
      a_{i j,k}&=\lambda_i g_{j k} + \lambda_j g_{i k}\\
      \lambda_{i,j}&=\mu g_{i j}-a_{i j}\\
      \mu_{,i}&=-2\lambda_i
    \end{array}
  \right.
\end{gather}

\begin{thm}[\cite{Mounoud2010}]
  \label{thm:parallel}
  If a symmetric tensor field $a_{ij}$ on  $(M,g)$  satisfies  \eqref{eq:mg-a}, 
  then the $(0,2)$-tensor  field $A$ on  $(\wt M, \tg)$ defined  in the local coordinates $(r,x)$ 
  by the following (symmetric) matrix:
  \begin{gather}
    \label{eq:A}
    A = \left(
      \begin{array}{c|ccc}
        \mu(x) & -r\lambda_1(x) &\dots &-r\lambda_n(x)\\
        \hline 
        -r\lambda_1(x)& & & \\
        \vdots& & r^2 a(x)& \\
        -r\lambda_n(x)& & & 
      \end{array}
    \right),
  \end{gather}
  is parallel  with respect to the Levi-Civita connection
  of $\tg$.

 Moreover, if a symmetric $(0,2)$-tensor $A_{ij}$ on $\wt M$ is parallel, then in the  cone coordinates it has the
  form~\eqref{eq:A}, where  $(a_{ij}, \lambda_i, \mu)$ 
  satisfy~\eqref{eq:mg-a}.
\end{thm}

\begin{proof}  This is an easy exercise (a straightforward way  to do this exercise is to write down   the condition that a symmetric parallel $(0,2)$-tensor field  on the cone is parallel, and compare it with \eqref{eq:mg-a}).
\end{proof}

\section{Proof of  Theorem~\ref{cov}.}
\label{sec:parallel_on_the_cone}
\subsection{Plan of the proof.}

We consider the cone $(\wt M, \tg) $  of dimension $n+1\ge 4$ over connected  simply connected  $(M, g)$. 
For every  $(0,2)$-tensor  field $A_{ij}$ on $\wt M$  we  
consider  the  $(1,1)$-tensor field  $L= L^i_j$ given 
by  \begin{equation} \label{defL}  A(., .)= \tg(L., .), \textrm{ i.e., in coordinates\ } L^i_j=\tg^{i k }A_{k j}. \end{equation} We will view $L$ as a field of endomorphisms of $T\wt M$.  If  $A$ is parallel and  symmetric, $L$ is 
  parallel and selfadjoint, and vise versa. 

Take $p\in \wt M$ and  consider the maximal orhtogonal   decomposition   of the  tangent space $T_p \wt M$ into
the direct sum of nondegenerate subspaces invariant w.r.t.  the action
of the holonomy group: \begin{equation} \label{-3} T_p \wt M=V_0\oplus V_1\oplus \dots \oplus V_\ell.\end{equation}
We assume that  $V_0$  is flat, in the sense that the holonomy group acts
trivially on $V_0$, and    that the decomposition is \emph{maximal}, i.e., 
 that each subspace $V_a,a\geq 1$, has no invariant
$\tg$-nondegenerate subspaces, and, therefore, cannot be decomposed
further.

We denote by  $\overset{(\alpha)}P = \overset{(\alpha)}P{}^j_i$  the orthogonal projector onto $V_\alpha,
\alpha=0,\dots , \ell $.  $\overset{(\alpha)}P$   is   selfadjoint and is 
preserved by the action of the holonomy group.   It corresponds to $g_a$ from Theorem \ref{cov} via \eqref{defL}. 
Clearly, $(\overset{(0)}P+\dots+\overset{(\ell)}P)=\Id$. 
We consider the following 
decomposition of  $L$:
\begin{gather}
  \label{eq:Ldecomposition}
  L=(\overset{(0)}P+\dots+\overset{(\ell)}P)\,  L\, 
  (\overset{(0)}P+\dots+\overset{(\ell)}P)=\sum_{a,b=0}^\ell  \overset{(a)}P L
  \overset{(b)}P.
\end{gather}

Each component $\overset{(a)}P L \overset{(b)}P$ is  an endomorphism invariant w.r.t. the holonomy group. Moreover, if $a=b$, then  it is self-adjoint.

The proof of the Theorem~\ref{cov}
contains two parts: first, in Lemma~\ref{thm:nondiag} we show that
each ``non-diagonal''  component $\overset{(a)}P L \overset{(b)}P$,  $a\neq b$, 
 is given by the quadratic combination of   vectors and  1-forms  invariant with respect to the holonomy group.  This part will be purely algebraic.     Then, in
Section~\ref{sec:nilpotent} 
we  describe  ``diagonal blocks'' $\overset{(a)}P L \overset{(a)}P$ and show that they are combinations of $\overset{(a)}P$ and  quadratic combination of vectors and  1-forms  invariant with respect to the holonomy group.   These two parts imply Theorem~\ref{cov}, we explain it  in  Section \ref{proof:thmcov}.

\subsection{Proof for ``non-diagonal''  components.  }\label{sec:Ldecomposition}

\begin{lem}
\label{thm:nondiag} In the notation above, 
let $L'=\overset{(a)}P L \overset{(b)}P$, $a\ne b$, be a  
``non-diagonal''  component of $L$ 
(invariant w.r.t. to the holonomy group). Then,
  $L'=\sum_{i,j} c_{ij} \tau_i^*\otimes \tau_j,$ where
  $\tau_s\in  T_p\wt M$ are certain  vectors invariant w.r.t. to the holonomy group,  $\tau^*_s \in  T^*_p\wt M$ are certain   $1$-forms invariant w.r.t. to the holonomy group, and $c_{ij} \in \R$ are constants.
\end{lem}
\begin{proof}
  Let $\bar u_i,\dots,  \bar u_k$ be a  basis of $\Im L'\subset V_a$
  and $\bar v_1, \dots \bar v_r$ be a  basis in $V_b$  such that  $L' \bar
  v_s =\bar u_s$ for $s=1,...,k$ and 
   and $\bar v_{k+1},\dots \bar v_r\in \ker L'$. Then, 
  \begin{equation} \label{-4} L'=\sum_{i\leq k} \bar u_i \otimes \bar v^*_i,\end{equation} where $v^*_i$ are the 1-forms dual to $v_i$.
It is known that   the holonomy group $H_p$ is the direct product of the subgroups
  $H_0\times  \dots  \times  H_\ell$ such that  each $H_\alpha$ acts trivially
  on all $V_\beta$ such that  $\beta\neq\alpha$.
  Since $L'(\wt M)\subset V(a)$, for each $h=h_0\cdot...\cdot h_l\in
  H$, $h_\alpha \in H_\alpha$,  and for any  $v\in T_p\wt M$ we have
  $$h\,\underbrace{L'(v)}^{\in V_a}=h_a \overset{(a)}P L \overset{(b)}P
  (v)=\overset{(a)}P L h_a(\underbrace{\overset{(b)}P (v)}^{\in V_b}) =
  \overset{(a)}P L \overset{(b)}P (v)=L'( v)$$ Thus, all $\bar u_i
  \in \Im L'$ are invariant with respect to the action of $H$.

  Similarly, consider the dual endomorphism $L'^*=\overset{(b)}{P^*} L^*
  \overset{(a)}{P^*}: V_a^* \to V_b^*$ and the action of the holonomy
  group $H$ on the dual decomposition:
  $$h\,\underbrace{L'^*(u^*)}^{\in V_b^*}=h_b \overset{(b)}{P^*} L^* \overset{(a)}{P^*}
  (u^*)=\overset{(b)}{P^*} L^* h_b(\underbrace{\overset{(a)}{P^*}
    (u^*)}^{\in V^*_a}) = \overset{(b)}{P^*} L^* \overset{(a)}{P^*}
  (u^*)=L'^*(u^*)$$
  Thus, $v_s^*=L'^*(u^*_s),s\leq k$ is invariant with respect to action of the
  holonomy group. Thus,  all $v_i$ and $v^*_i$ from \eqref{-4} are invariant w.r.t. holonomy group.
\end{proof}

\subsection{Parallel symmetric tensor fields on  indecomposable
  pseudo-Riemannian manifolds.}
\label{sec:nilpotent}

In this section we deal
with the ``diagonal'' components $\overset{(a)}P L \overset{(a)}P$
of an arbitrary parallel self-adjoint tensor $L$.
We take $a\ge 1$ and   denote by  $M_a$  the $k_a$-dimensional integral submanifold 
corresponding to the subspace $V_a$ and by $g_a$ the restriction of the metric to it. Clearly,   
\begin{enumerate}
\item  $M_a$ is indecomposable;
\item by Lemma~\ref{thm:split}, it admits a  function satisfying   \eqref{eq:hom}
\item if  the signature of the initial metric
  $g$ is  riemannian or lorentzian,  then, by
  Lemma~\ref{thm:two_cases}, the cone metric $\tg$ has
  signature $(1,n)$, $(n-1,2)$, or the riemannian signature $(0,n+1)$. 
  Thus, the restriction $g_a$ of the
  cone metric to each component $M_a$ is either Riemannian, Lorentzian
  or has signature $(k_a-2,2)$;
\item The restriction of $\overset{(a)}P L \overset{(a)}P$  to $M_a$ 
is a well-defined
 parallel  selfadjoint $(1,1)$-tensor field on $M_a$.
\end{enumerate}

For readability we  ``forget'' the index   $a$ and    denote the manifold  $M_a$, the
metric $g_a$ on it and the restriction of  $\overset{(a)}P L \overset{(a)}P$ to it by 
 $\wt M$, $g$ and $L$ and assume that $k_a=\dim M_a=n+1$; they enjoy the properties (1--4) above. 
 
The goal of the next two sections  will be to prove that 
 $L$   and the curvature tensor $R$  fulfill

 \begin{equation} \label{kernel} L^i_{p}R^p_{\ jk\ell}=0. \end{equation} 
In order to prove this result we will use  the following property of parallel $(1,1)$-tensor fields: 
    \begin{equation} \label{ricci} L^i_{p}R^p_{\ jk\ell}=R^i_{\ p k\ell}L^p_j.\end{equation}

    In order to prove \eqref{ricci},  we use that  for  the vector  fields
    $X=\partial_k,Y=\partial_\ell,Z=\partial_j$, in view of  $[X,Y]=0$, we have 
    $$R(X,Y)(Z)=\nabla_X \nabla_Y Z - \nabla_Y \nabla_X Z$$
    Applying the (1,1)-tensor $L$ viewed as an endomorphism  an using that it is parallel   we obtain
    $$
      L(R(X,Y)(Z))=L(\nabla_X \nabla_Y Z - \nabla_Y \nabla_X Z)=\nabla_X \nabla_Y L(Z) - \nabla_Y \nabla_X L(Z)=R(X,Y)(L(Z))
     $$ which is equivalent to \eqref{ricci}.

Besides, we will use  that for a solution $v$ of \eqref{eq:hom}  and for the corresponding vector field 
 $\vec v := v_{,}^{\  i}$    we have $L^k  \vec v  \ne 0$ almost everywhere provided $L^k:= \underbrace{L\circ ...\circ L}_{\textrm{$k$ times}} $  is not identically zero. 
Indeed, the existence of a solution of \eqref{eq:hom} implies that $g$ or $-g$ 
 is a cone metric (in a neighbohood of almost every point). Then,   by \eqref{eq:A}, $L^i_j  \vec v^j= \pm \lambda^i$. Now, by Theorem \ref{thm:mg+Ba}, the solutions $(a, \lambda, \mu)$ satisfy \eqref{eq:mg+Ba} and  by assumptions we have $B= \pm 1$. Therefore, if  $\lambda^i $ is  zero at every point of  an open subset,  $L$ is proportional to $\delta^i_j$ 
 in this subset implying it is $\delta^i_j$ everywhere.

   \subsubsection{ Possible  Jordan forms of $L$.} \label{possible_jordan_forms}

We   first recall the following  theorem from linear algebra: 
\begin{thm}[\cite{LancRodm}, Theorem 12.2]
  \label{thm:pairofmatrices}
  Let $g$ be a symmetric bilinear nondegenerate form on a $n$-dimensional real linear vector space $V$, and 
  let   $L$ be a   $g$-self-adjoint endomorphism of $V$.  Then there exists
  a basis in $V$ such that in this basis the matrices of $g$ and $L$ have the  blockdiagonal form \begin{gather}
  L=\left(
    \begin{array}{ccc|ccc}
      J_{l_1}&&&&\\
      &\ddots&&&\\
      &&J_{l_p}&&\\
      \hline
      &&&J_{2 m_1}&\\
      &&&&\ddots\\
      &&&&&J_{2m_q}
    \end{array}
  \right),\\ g=\left(
    \begin{array}{ccc|ccc}
      \varepsilon_1 F_{l_1}&&&&\\
      &\ddots&&&\\
      &&\varepsilon_p F_{l_p}&&\\
      \hline
      &&&F_{2 m_1}&\\
      &&&&\ddots\\
      &&&&&F_{2m_q}
    \end{array}
  \right),
\end{gather}
where $J_{k_i}$, $i=1,\dots,p$,  are  the $k_i$-dimensional elementary Jordan blocks with  real eigenvalues, 
$J_{2 m_i}$, $i=1,\dots,q$,   are 
is the $2m_i$-dimensional elementary (real) Jordan block with  complex  eigenvalues, $F_k$ are  the 
  $k\times k$-dimensional symmetric matrices   of the form 
  \begin{gather}
    \label{eq:elementary_block}
    F_m=\left(
      \begin{array}{ccccc}
        &&&&1\\
        &&&1&\\
        &&\iddots&&\\
        &1&&&\\
        1&&&&
      \end{array}
      \right), 
  \end{gather} and $\varepsilon_i \in \{1,-1\}$. 
\end{thm}

It is easy to see that  
the $k\times k$-dimensional matrices  $\varepsilon F_k$ (viewed as bilinear forms on $\R^k$) have 
the signature $(\lfloor\frac{k}{2}\rfloor,\lfloor\frac{k+1}{2}\rfloor)$ or
$(\lfloor\frac{k+1}{2}\rfloor,\lfloor\frac{k}{2}\rfloor)$ depending on the sigh of $\varepsilon$, and that each  block $F_{2m_i}$  has signature $(m_i,m_i)$

We will apply this theorem to our metric $g$ which has   signature   is $(1,n)$, $(n-1, 2)$ or $(0,n+1)$. Moreover, since our  $L$ is invariant w.r.t. the holonomy group, it can not have two different real eigenvalues, or two different pairs of complex-conjugate eigenvalues, or  simultaneously a real eigenavalue  and a complex eigenvalues. We therefore have: 

\begin{cor} Under our assumptions,   
if $L$ has a real eigenvalue, then it is its only eigenvalue and the Jordan form of $L$ has at most two Jordan blocks of dimension $\ge 2$.  If $L$ has a complex nonreal eigenvalue, then  
the dimension of $\wt M$ is 4.  
\end{cor}

Note that a (1,1)-tensor $L$ is parallel and selfadjoint if and only if (for any constants $c_1\ne 0$, $c_2$) the  (1,1)-tensor $c_1 L + c_2 \textrm{Id}$ is parallel and selfadjoint.  Thus, if $L$ has a real eigenvalue, then without loss of generality we can assume that $L$ is nilpotent. If $L$ has a complex eigenvalue, then without loss of generality in a certain  basic in $T_p\wt M$ the tensor $L$ and the metric $g$ are given by \eqref{-8}, \eqref{-9}.

\subsubsection{Proof of \eqref{kernel} for  self-adjoint nilpotent endomorphisms    with at most two  Jordan
  block of dimension $\ge 2$ on  cone manifolds.}
\label{sec:two_blocks}

 The proof is a purely linear algebraic: we derive \eqref{kernel} from \eqref{ricci}, from the assumption that $L$ is nilpotent with   at most two Jordan blocks of dimension $\ge 2$, and from   the existence of a vector $\vec v$ such that 
  $L^r\vec v\ne 0$ for all $r$ such that $L^r \ne 0$. All these conditions are  fulfilled 
   at every point of a certain everywhere dense open subset of $M$; clearly, if \eqref{kernel} is fulfilled 
   at every point of a certain everywhere dense open subset of $M$, it is fulfilled everywhere.  
  
  We consider  a generic  point $p\in M$ and 
 take a basis $e_1,...,e_n$ in $T_p\wt M$ 
such that in this basis $L$ has the block-diagonal form
 \begin{gather}
    \label{eq:block-diagonal-bis} 
    L=\left(
    \begin{array}{ccccc}
      J_{k} &       & &  &\\
            & J_{m} &  & &\\
      &&0&\dots &0  \\
      &&\vdots&\ddots & \vdots \\
      &&0& \dots& 0 \\
    \end{array}\right),
  \end{gather}
where $J_k$  and $J_m$ are $k\times k$  and $m\times m $ dimensional Jordan blocks. Then, 
$L(e_i)= e_{i-1}$ for $i=2,...,k, k+2,...,k+m$ and $L(e_i)= 0$ for $i=1,k+1,k+m+1,...,n+1$. 
We assume $k\ge m$ and allow   $m=1$.   Because of $L^{k-1} \ne 0$, we have $L^{k-1}  \vec v  \ne 0$ and therefore $\vec v$  has the maximal height;  without loss of generality we may  assume   $e_k= \vec v$.

  We take two arbitrary vectors $X,Y\in T_p  M$ and 
   the      $g$-skew-selfadjoint endomorphism  $\tilde R:= R(X,Y)= R^i_{\ j k \ell} X^k Y^\ell  $ on $T\wt M$. Then,  the condition \eqref{ricci} implies that   $L$ and $\tilde R$
  commute as linear endomorphisms.    Let us consider 
  the  bilinear form $g( L\tilde R\,\cdot,\cdot)$ 
   and show that it vanishes. 
   
   Since for any $u$   and $w$  and for any $r\in \mathbb{N}$ 
    we have 
    \begin{equation} \label{-6}  g(  L^r \tilde R u,w)= g( \tilde R u,L^r w)=-g( u,\tilde R
   L^r w)=- g(\tilde R L^r  w,u)=- g( L^r  \tilde R w,u),\end{equation} we see that  the bilinear form   $ g( L^r\tilde R\,\cdot,\cdot)$  is skew-symmetric; in particular $ g(  L^r \tilde R u,u)= 0$ for all $u$.

 We show 
 \begin{equation} \label{-7} g( \tilde R L e_i, e_j)=0  \textrm{ \ for all $i,j =1,...,n+1$}. \end{equation}

 For $i = k+m+1,...,n+1$ and arbitrary $j$ we have $L(e_i)=  0$   so \eqref{-7} trivially holds. 
 For $i= 1,...,k$ and arbitrary $j$  we have 
 $$
 g (\tilde R L e_i, e_j) = g(\tilde R L^{k-i+1} \vec v, e_j)= g(\tilde  L^{k-i+1} \tilde R \vec v,  e_j)= g(0,  e_j)=  0,$$ 
 so \eqref{-7}  holds as well. 
Since $g(\tilde   R  L\cdot , \cdot ) $ is  skew-symmetric, we  also have  \eqref{-7} for  $j=  1,...,k, k+m+1,...,n+1$ and  arbitrary $i$.  
Now, for  the remaining pairs of indexes $i, j = k+1,...,k+m$,   we have 
$$
 \tg (\tilde R L e_i, e_j) = \tg(\tilde R L^{k+m-i+1} e_{k+m},  L^{k+m-j} e_{k+m})= \tg(\tilde R  L^{2k+2m -i-j+1}    e_{k+m},  e_{k+m})= 0.$$ 
 Thus, $g( \tilde R L \cdot , \cdot )\equiv 0$ implying $\tilde R L=0$ as we claimed.

\subsubsection{Two interesting (counter)examples.} \label{counter}

The next two examples show that the assumptions in Section~\ref{sec:two_blocks}
that $\hat g$ is a cone metric and that  $L$ has at most two  nontrivial Jordan
blocks  are important.

The first example is based on the description of nilpotent parallel   symmetric $(0,2)$-tensor fields due to  Solodovnikov~\cite{SK} and
Boubel~\cite{Boubel}.  In order to produce  the second example,  we applied  the construction from  
\cite[Theorem 3.3]{Mounoud}  to $g$ and $L$  from the first example. 

\begin{ex} \label{ex1} 
  In coordinates $(x_1,x_2,x_3,x_4)$ on $U\subset \R^4$ we consider a metric $g$ and a $(1,1)$-tensor field $L$ given by 
  $$
  g=\left(
    \begin{array}{cccc}
      0&0&x_3 x_4&0\\
      0&0&0&x_3 x_4\\
      x_3 x_4&0&x_1 x_4+x_2 x_3&0\\
      0&x_3 x_4&0&x_1 x_4+x_2 x_3
    \end{array}
  \right),
  L=\left(
    \begin{array}{cccc}
      0&0&1&0\\
      0&0&0&1\\
      0&0&0&0\\
      0&0&0&0
    \end{array}
  \right)$$

By direct computation it is easy to  show that $L$ is parallel  and self-adjoint, 
  while $L^i_p R^p_{jkm}\ne 0$  (because,   for example, $L^1_p R^p_{434}\neq 0$).
\end{ex}

\begin{ex} \label{ex2} 
  We denote by    $(r, s,x_1,x_2,x_3,x_4)$ the coordinates
  on $U\subset \R^6$ consider the following
  function $$F(r,s,x_1,x_2,x_3,x_4)=r^2 e^{2s}(x_1 x_4+x_2 x_3)+r^2
  x_3 x_4. $$ Then, we put
  \begin{gather}
    \tg=
    \left(
    \begin{array}{cc|cccc}
      1&0&0&0&0&0\\
      0&-r^2&0&0&0&0\\
      \hline
      0&0&0&0&r^2 e^{2s} x_3 x_4&0\\
      0&0&0&0&0&r^2 e^{2s} x_3 x_4\\
      0&0&r^2 e^{2s} x_3 x_4&0&F&0\\
      0&0&0&r^2 e^{2s}x_3 x_4&0&F
    \end{array}
  \right),\\
  \widehat L=e^{2s}\left(
    \begin{array}{cc|cccc}
      1&r&0&0&0&0\\
      -\frac{1}{r}&- 1&0&0&0&0\\
      \hline
      0&0&0&0&1&0\\
      0&0&0&0&0&1\\
      0&0&0&0&0&0\\
      0&0&0&0&0&0
    \end{array}
  \right)
\end{gather}
Evidently, $\tg$ is a cone metric. By direct computation one can prove 
 that $\widehat L$ is a parallel self-adjoint tensor field with respect to $\tg$, which is
nilpotent and has three $2$-dimensional Jordan blocks, and that    $L^i_p R^p_{jkm}\ne 0$.
\end{ex}

The last  example shows also that 
 the assumption on the signature  of $\hat g$   in
Theorem~\ref{cov} is important.

\subsubsection{ Indecomposable  blocks can not have complex eigenvalues of $L$.}
\label{sec:complex}
Let us now consider the case when  (parallel, selfadjoint)  $L$  on the 
indecomposable cone manifold $\wt M$  has  
two  complex conjugate eigenvalues.    Then, as we explained in Section \ref{possible_jordan_forms},  we may think that in a certain basis the matrices of $g$ and  $L$ are as below 
  \begin{gather}
    L=\left(
      \begin{array}{cc|cc}
        0 & 1 & 0&0 \\
        -1 & 0 & 0&0\\
        \hline
        0&0  &0 &1\\
        0&0  &-1 &0\\
      \end{array}
    \right), \mbox{ \   } g= \left(
      \begin{array}{cc|cc}
        0 & 1 & 0&0 \\
        1 & 0 & 0&0\\
        \hline
        0&0  &0 &1\\
        0&0  &1 &0\\
      \end{array}
      \right) \label{-8}\\
    L=\left(
      \begin{array}{cc|cc}
        0 & 1 & 1&0 \\
        -1 & 0 & 0&1\\
        \hline
        0&0  &0 &1\\
        0&0  &-1 &0\\
      \end{array}
    \right),\mbox{ \  }g= \left(
      \begin{array}{cc|cc}
        0 & 0 & 0&1 \\
        0 & 0 & 1&0\\
        \hline
        0&1  &0 &0\\
        1&0  &0 &0\\
      \end{array} \label{-9}
      \right)
  \end{gather}

We again consider the      $g$-skew-selfadjoint endomorphism  $\tilde R:= R(X,Y)= R^i_{\ j k \ell} X^k Y^\ell  $, where 
 $X,Y\in T_p  \wt M$ are arbitrary vectors. 
  Then,  the condition \eqref{ricci} implies that   $L$ and $\tilde R$
  commute as linear endomorphsims. 
   The matrix $\tilde R$ satisfies therefore  the relations 
\begin{equation} \label{-10}
  \tilde RL - \tilde LR=0 \textrm{ and }  g \tilde R + \tilde R^t g= 0. 
  \end{equation} 
  
  Suppose now  $L,g$ are  as in \eqref{-8}.   Then, \eqref{-10}  is a system of linear equations on the components of $\tilde R$.   Solving it (which is an easy exercise in linear algebra) 
   we obtain  that in this basis $\tilde R$ has the form
  $$
  \left[ \begin {array}{cccc} 0&0&-\tilde R_{{4,2}}&\tilde R_{{4,1}}
\\ \noalign{\medskip}0&0&-\tilde R_{{4,1}}&-\tilde R_{{4,2}}\\ \noalign{\medskip}
\tilde R_{{4,2}}&-\tilde R_{{4,1}}&0&0\\ \noalign{\medskip} {\tilde R_{4,1}}&{\tilde R_{4,2}}&0&0
\end {array} \right].$$ We see that $\tilde R$ is nondegenerate unless $\tilde R_{{4,2}}= \tilde R_{{4,1}}=0$. 
But it is degenerate since $\tilde R\vec v =  0$. Then,   $\tilde R_{{4,2}}= \tilde R_{{4,1}}=0$. Thus, for any $X,Y$ we have  $\tilde R:= R(X,Y) =0$ implying the metric is flat. 
 
  Suppose now  $L,g$ are  as in \eqref{-9}.  In this case  the equations   
\eqref{-10} already imply that $\tilde R= 0$. 
 Thus, also in this case,  for any $X,Y$, we have  $\tilde R:= R(X,Y) =0$ implying the metric is flat.

\subsubsection{The existence of parallel vector fields provided \eqref{kernel}.}

\label{sec:kernel}

In   Sections  \ref{possible_jordan_forms}, \ref{sec:two_blocks},  \ref{sec:complex} 
 we have shown that under the  general assumptions of Section  \ref{sec:nilpotent} 
   each selfadjoint  parallel  tensor is the sum of $\const \cdot \Id$ and a (parallel, selfadjoint)  tensor $L$   
satisfying 
$L^i_s R^s_{\ jkl}=0.$ Then,  the goal of Section  \ref{sec:nilpotent}, i.e., 
 the ``diagonal'' part of Theorem \ref{cov} follows from

\begin{lem}
  \label{thm:kernelR}
  Let $L$ be a parallel   $(1,1)$-tensor on a connected simply-connected  $(M,g)$ satisfying 
    $L^i_s R^s_{\
    jkl}=0 $ (where $R$ is the curvature tensor of $g$). 
    Then,  there exists $r= \textrm{rank}(L)$ linearly independent parallel vector fields of $M$ such that at every point they lie in the image of $L$.  
\end{lem}
\begin{proof}
At every point $p\in M$ we consider  the subspace 
$D_p:= \textrm{\it Image}(L) \subseteq T_pM $.
 Since the tensor $L$ is parallel, its rank is constant and $D$ is a smooth distribution. Since $L$ is parallel, $D_p$ is integrable and totally geodesic. Then, the restriction of the Levi-Civita connection of $g$ to  the  $D_p$ (considered as a subbundle of the tangent bundle) is well defined, and its curvature is the restriction of the curvature tensor $R$ to $D$.
  The condition  $L^i_s R^s_{\
    jkl}=0 $ implies that  the curvature of the restriction of the $g$-Levi-Civita 
    connection to $D_p$ is zero, so the  connection on the subbundle $D$ is flat. Then, each vector $v(p)\in D_p$ can be extended to a parallel section in $D$. 
\end{proof}

\subsection{Collecting all facts: proof of Theorem~\ref{cov}.} \label{proof:thmcov}

Let $\wt M$ be a  $n+1\ge 4$-dimensional cone manifold of signature 
$(0,n+1)$, $(1,n)$ or $(n-1,2)$. Let $T\wt M = V_0\oplus V_1\oplus \dots\oplus V_\ell$ be a 
maximal   orthogonal nondegenerate decomposition  invariant w.r.t. holonomy group 
and $L$ be a   
selfadjoint endomorphism invariant w.r.t. holonomy group.

We consider the  decomposition~\eqref{eq:Ldecomposition} of $L$ into the
sum of orthogonal projectors and regroup the summands to obtain:
\begin{gather}
  L=\underbrace{\overset{(0)}P L \overset{(0)}P}_{(A)} + \sum_{a\neq
b} \underbrace{\overset{(a)}P L \overset{(b)}P}_{(B)}+\sum_{a\geq 1}
\underbrace{\overset{(a)}P L \overset{(a)}P}_{(C)}
\end{gather}

It is sufficient to show that each   term $(A), (B), (C)$ is a linear combination of projectors $\overset{(\alpha)}{P}$ and an  endomorphism of the form $\sum c_{ij } \tau_i \otimes \tau^*_i$, where $\tau_i$ and $\tau^*_j$ are vectors and 1-forms invariant with respect to the holonomy group. For the $(A)$-component  it is nothing to prove: every endomorphism from $V_0$ to $V_0$  has this form. For the $(B)$-components, we have proved this in Lemma~\ref{thm:nondiag}. For the $(C)$-components,   this follows from 
Lemma \ref{thm:kernelR}. 
 Theorem~\ref{cov} is proven.

\section{Proof of  Theorem~\ref{cov1}.}
\label{sec:proof_of_the_main_result}

As in the proof of  Theorem~\ref{cov}, we  consider the
 maximal orthogonal  decomposition $$T\wt M=V_0\oplus\dots \oplus V_\ell,$$ where
 $V_0$ is a   $\hat g$-nondegenerate  subspace of maximal simension  such that 
 the holonomy group  acts trivially on it and $V_\alpha,1\leq\alpha \leq \ell$ are  $\hat g$-nondegenerate   subspaces invariant w.r.t. the holonomy group. We denote by $k$ the  dimension of the subspace  where the holonomy group acts trivially, i.e.,  the  number of 
linearly independent parallel  vector fields of $\hat g$.
 We need   to prove that 
 the possible values $(k,\ell)$ are  $(k=0,\dots ,n-2, \ell=1,\dots, \lfloor
  \tfrac{n-k+1}{3}\rfloor)$..

 We first prove  $k\le dim(\wt M)-3 = n-2$. Indeed, suppose we have $n-1$ parallel vector fields.  By Lemma \ref{thm:cone-not-parallel}, parallel vector fields $u$  and the cone vector field $\vec v$ are linearly independent at points such that $R_{ijkm}\ne 0$.  We take a basis at the tangent space $T_p\wt M$ (for almost every point  $p$ such that $R_{ijkm}\ne 0$)
 such that the first $n-1$ vectors of the basis are the parallel vector fields, the $n$th 
 vector     is the vector $\vec v$. 
As we have shown in the proof of Lemma \ref{thm:cone-not-parallel}, the parallel vector fields $u$ and the cone vector field $\vec v$ 
satisfy 
$$
\vec v^s R^i_{\ s jm} = u^s R^i_{\ s jm} =0. 
$$
Then, in this basis, the components $R_{ijms}$ such that at least one of the numbers $i,j,m,s$ is not $n+1$ are zero. The remaining component  $R_{ijms}$ with $i=j=m=s=n+1$ is also zero in view of the symmetries of the curvature tensor. Finally, $R^i_{\  jms}\equiv 0$  which contradicts the assumptions of Theorem \ref{cov1}.

 Let us now show that $\ell$  is at  most $\lfloor
  \tfrac{n-k+1}{3}\rfloor$.

We denote by $U $ the subspace of  $T_p\wt M$ 
such that the holonomy  group acts trivially on it.
For $\alpha=0, \dots,   \ell$, we  put   $s_\alpha=\dim V_\alpha$ and for $\alpha=1, \dots,   \ell$ we put 
$r_\alpha=\dim U\cap V_\alpha$. Evidently,  $$\dim \wt M=n+1=s_0+s_1+\dots +  s_\ell
\mbox{ and } k=s_0+r_1+\dots+r_\ell.$$ 
Next,  let us show for every   $\alpha =1,...,\ell$  we have  $s_\alpha\geq r_\alpha+3$.

Let $\overset{(\alpha)}{R}{}^i_{\ jkl}$ be the restriction of the  Riemannian curvature
to $V_\alpha$. Then, 
for every $u\in U\cap V_\alpha$ we evidently have $ u_i \overset{(\alpha)}{R}{}^i_{\ jkl}=0.$
Moreover, as we have shown in Lemma \ref{thm:split},  for the cone vector field $\vec v:= v_{,}^{\ i}$,  the vector
$\overset{(\alpha)}{\vec v}{}^i=\overset{(\alpha)}{P}{}\vec v$   is the gradient of a certain function 
$\overset{(\alpha)}v$ satisfying \eqref{eq:hom} (w.r.t. to $g_\alpha$), see  Lemma~\ref{thm:hom},  and therefore is nonzero at almost all  points  and also  satisfies 
$$ \overset{(\alpha)}{v}{}^j \overset{(\alpha)}{R}{}^i_{\ jkl}=0.$$

By Lemma~\ref{thm:cone-not-parallel}, $\overset{(\alpha)}{v}{}_i$ is
linearly independent of  the space $U\cap V_\alpha$. Thus, at least $(r_\alpha +
1)$ linearly independent vectors  $u\in V_\alpha$ (at the tangent space of  almost every point)
 satisfy  ${u^j} \overset{(\alpha)}{R}{}^i_{\ jkl}=0$. 

Suppose $\dim V_\alpha= s_\alpha \leq r_\alpha +2$. Then, since 
$\overset{(\alpha)}{R}{}^{i}_{\ jkm}$ 
is $\hat g$-skew-symmetric with respect to the first two indexes $i,j$,
 it must be zero, which contradicts the assumption.  
Therefore, $s_\alpha\ge  r_\alpha+3$. Combining this with 
$n+1=s_0+s_1+\dots+s_\ell \geq s_0+ (r_1+3)+\dots+(r_\ell+3)=k+3\, \ell, $
we obtain $\ell\leq \lfloor \frac{n+1-k}{3}\rfloor$. Theorem~\ref{cov1} is proven.

\begin{rem}
As we explained in Section \ref{sec:Bneq0}, just proved  Theorem \ref{cov1} implies Theorem \ref{th:main} under the additional assumption $B= B(g)\ne 0$. 
\end{rem}

\section{Proof of Theorem \ref{th:main} if  $B=0$ and there exists a solution  $(a,\lambda,\mu)$ with $\mu
\neq 0$.}
\label{sec:zeroB}

\subsection{Scheme of the proof.} We  reduce this case to   the already proven case when $B\ne 0$. The reduction is as follows: 
in Section~\ref{sec:mu_neq_0} we show that in any open subset
on $M$ with compact closure there exists a geodesically equivalent
metric $\bar g$  that  is arbitrary close to $g$ and such that  $\bar B= B(\bar g)\ne 0$. Since geodesically equivalent metrics evidently have the same degree of mobility, we obtain that for any 
connected simply connected 
neighborhood  $U\subseteq M$ 
with compact support the   degree  of mobility of $g_{|U}$ is as in Theorem \ref{th:main}.  
Having this,   in Section~\ref{sec:noncompact} we show  that  on the whole manifold 
 the degree of mobility is as we claim in Theorem \ref{th:main} .

The remaining case, when the extended system    does not admit solutions with
$\mu\neq 0$,  will be considered  in Section~\ref{sec:B=0_isotropic}.

\subsection{How  $B$ changes if we change the metric in the projective class.}\label{how}

\begin{lem}
\label{thm:newB}
Let $g$ and $\bg$ be two nonproportional geodesically equivalent
metrics with degree of mobility $D(g)=D(\bg)\geq 3$ and let $\phi, a,
\lambda$ and $\mu$ be as in Sections~\ref{standard}, \ref{standard1}. Then, the constant 
$\bar B = B(\bar g)$  
is equal to
  \begin{gather}
    \label{eq:barB}
    \bar B = -e^{-2\phi}(\mu + \phi_p \lambda^p)
  \end{gather}
\end{lem}
\begin{proof}
  
  Since $D(g)\geq 3$, for every $\bg$ there exists  a  triple
  $(a_{ij},\lambda_i,\mu)$  satisfying ~\eqref{eq:mg+Ba} such that  
  \begin{gather}
    \label{eq:ae}
    a_{i j} = e^{2\phi} g_{i p} \bar g^{p q} g_{q
      j}, \ \ \lambda_k= \frac{1}{2} \partial_k\left( a_{pq} g^{pq}\right).
  \end{gather}
  By direct computation,
  \begin{multline}
    \label{eq:lambdae}
    \lambda_k{=}
   \frac{1}{2} \partial_k\left( a_{ij} g^{ij}\right)\overset{\eqref{eq:ae}}{=}
    e^{2\phi}\phi_k \bg^{p q}g_{p q} +e^{2\phi}
    g_{p q} \bg^{p q}_{\ \
      ,k}\overset{\eqref{LC}}{=}\\
    = e^{2\phi}\phi_k \bg^{p q}g_{p q}+
    \frac{1}{2}e^{-2\phi}g_{p q}(-2\phi_k\bg^{p q} -
    \phi_s \bg^{p s}\delta^q_k- \phi_s \bg^{q
      s}\delta^p_k)=-e^{2\phi} \phi_p \bg^{p
      q}g_{q k}
  \end{multline}
  Let us calculate $\lambda_{i,j}$.
  \begin{multline}
    \lambda_{i,j}=\nabla_j\left(-e^{2\phi} \phi_p \bg^{p
        q}g_{q i}\right)={}-2 e^{2\phi}\phi_j \phi_p
    \bg^{p q}g_{q i} - e^{2\phi} \phi_{p,j}
    \bg^{p q}g_{q i} -e^{2\phi} \phi_p \bg^{p
      q}_{\ \ ,j}g_{q
      i} \overset{\eqref{LC}}{=}\\
    = {}-2 e^{2\phi}\phi_j \phi_p \bg^{p q}g_{q i} -
    e^{2\phi} \phi_{p,j} \bg^{p q}g_{q i}+ e^{2\phi}
    \phi_p g_{q i}(2\phi_j \bg^{p q}+\phi_s
    \bg^{s p}\delta^q_j+\phi_s \bg^{s q}\delta^p_j)=\\
    =-e^{2\phi}\phi_{p,j}\bg^{p q}g_{q i}-g_{i
      j}(\phi_p \phi_q \bg^{p
      q})+e^{2\phi}\phi_p \phi_j \bg^{p q} 
      g_{q i}
  \end{multline}

 Next, we  substitute $\lambda_{i,j}=\mu g_{ij}+B a_{ij}$ which is the second equation of 
 \eqref{eq:mg+Ba}  and rearrange the
  components to obtain 
  \begin{gather}
    (\mu + \phi_p \lambda^p)g_{ij}= e^{2\phi}\bg^{p
      q}g_{q i}(\phi_p \phi_j-\phi_{p,j}-B g_{p
      j})
\end{gather}
Multipling the equation by $e^{-2\phi} g^{i p}\bg_{p q}$ and renaming
the indices we obtain 
\begin{gather}
  \label{eq:phiB}
  e^{-2\phi}(\mu + \phi_p \lambda^p) \bg_{i j} + B g_{i j} =
  \phi_i \phi_j-\phi_{i,j}
\end{gather}

Let us now swap metrics $g$ and $\bg$ and rewrite~\eqref{eq:phiB} in
the form:
\begin{gather}
  e^{-2\bar \phi}(\bar \mu + \bar \phi_p \bar \lambda^p)
  g_{i j} + \bar B \bg_{i j} = \bar \phi_i \bar \phi_j-\bar \phi_{i:j}
\end{gather}
Here we denote all the components corresponding to the chosen metric
$\bg$ with bar and derivation with respect to the Levi-Civita  $\bar \nabla$ of $\bar g$  by 
semicolon. It is easy to see that $\bar \phi=-\phi$
and $$\phi_{i,j}=\phi_{i:j}+2\phi_i\phi_j$$

We substite $\phi_i \phi_j-\phi_{i,j}=-(\bar \phi_i \bar \phi_j-\bar
\phi_{i:j})$ in~\eqref{eq:phiB} to obtain:
\begin{gather}
  e^{-2\phi}(\mu + \phi_p \lambda^p) \bg_{i j} + B g_{i
    j} = -e^{-2\bar \phi}(\bar \mu + \bar \phi_p \bar \lambda^p)
  g_{i j} - \bar B \bg_{i j}
\end{gather}
Thus, 
$$ \left(e^{-2\phi}(\mu + \phi_p \lambda^p) + \bar B\right)
\bg_{i j}=\left(-e^{-2\bar \phi}(\bar \mu + \bar \phi_p \bar
  \lambda^p) -B \right)g_{ij}$$
By Weyl \cite{Weyl2}, $g$ and $\bg$ are   nonproportional at almost every points, so 
both scalar coefficients vanish and the formula~\eqref{eq:barB} is proven.
\end{proof}

\subsection{The local existence of a geodesically equivalent metric $\bar g$ 
  with $\bar B= B(\bar g)\neq 0$.}
\label{sec:mu_neq_0}

\begin{lem}
  \label{thm:mu_neq_0}
  Let $(M,g)$ be  a Lorentzian manifold with $D(g)\geq 3$. Assume  $B=0$
  and suppose  
  that the extended system~\eqref{eq:mg+Ba} admits a  solution
  $(a,\lambda,\mu)$ with $\mu\neq 0$.
 Let $U$ be an open subset in $M$ with compact closure.

  Then, there exists a metric $\bar g$  on $U$ that it  is 
  geodesically equivalent to the restriction $g|_U$,  such that the
  corresponding constant $\bar B:= B(\bar g) \neq 0$, and such that $\bar g$ is arbitrary close to $g$ in  the $C^2$-topology.
\end{lem}
\begin{proof}

  Since $B=0$, the extended system \eqref{eq:mg+Ba} reads 
  \begin{gather}
    \label{eq:mg} 
    \left\{
      \begin{array}{ll} 
        a_{i j,k}&=\lambda_i g_{j k} + \lambda_j g_{i
          k}\\
        \lambda_{i,j}&=\mu g_{i j}\\
        \mu_{,i}&=0.
      \end{array} \right.
  \end{gather}
  Thus,   $\mu$ is a
  constant. By assumption, there exists a
  solution $(a,\lambda,\mu)$ with $\mu\neq 0$. Without loss of
  generality we can assume $\mu=1$.

  Consider the one-parameter family   $(a_{ij}(t):=t
  \lambda_i \lambda_j + g_{ij}, \lambda_i(t):=t\lambda_i,\mu(t):=t\mu=t)$. It is easy to see that for each $t$ the triple  $(a(t), \lambda(t), \mu(t))$
  satisfies~\eqref{eq:mg}.
  
  Evidently, since $U$ has a compact closure, there exists
  (sufficiently small) $t_0>0$, such that for all $-t_0<t<t_0$ the  solution
  $a_{ij}(t)$ is nondegenerate everywhere on $U$ and the signature of
  the corresponding metric $\bar g(t)$ coincides with that of $g$.
    
  The triple $(a_{ij}(t),\lambda_i(t),\mu(t))$ determines the metric $\bar
  g(t)$ and the 1-form $\phi(t)$ on $U$. By Lemma~\ref{thm:newB}
  $B(t):= B(\bar g(t))=-e^{-2\phi(t)}(\mu(t) + \phi_p(t)
  \lambda^p(t)).$

  Our goal is to show that there exists $t$  such that  $\bar B(t)\neq
  0$. Since $e^{-2\phi(t)}> 0$, it is sufficient to prove that
  $B^\ast(t)=\mu(t) + \phi_p(t) \lambda^p(t) > 0$
  for a certain $t$.   Let us calculate the  $\tfrac{d}{dt}$-derivative of $B^*$  at  $t=0$:
  $$\left. \tfrac{d}{dt} \right|_{t=0} B^\ast(t) = \left. \frac{d}{dt}
  \right|_{t=0} \left( t +t \lambda^p\phi_p(t)\right)= 1+\phi_p(0)
  \lambda^p=1\neq 0.$$

  Since the  smooth function $B^\ast(t)$ has non-zero derivative at the
  point $t=0$, there exists sufficiently small positive $t<t_0$ such that
  $B^\ast(t)$ and, therefore, $B(t)$ is not zero. Then,  the metric
  $\bar g=\bar g(t)$ satisfies the requirements.
\end{proof}

\subsection{Transition ``local''$\longrightarrow$ ``on a simply-connected manifold''.}
\label{sec:noncompact}

\begin{lem}
  \label{thm:noncompact}
  Let $(M,g)$ be a connected pseudo-Riemannian manifold, and
  $M=\cup_{s=1}^\infty M_s$, where $M_s$ are open connected subsets in
  $M$ and $M_s\subset M_{s+1}$. Denote by $g_s$ the restriction of $g$
  to  $M_s$. Then,  there exists $k$ such that for every $k'>k$ we have  $D(g)=D(g_{k'})$.
\end{lem}
\begin{proof}
  Evidently, for every solution $a\in Sol(g_s)$, its restriction to
  $M_{s'}\subset M_s$ with $s'<s$ is a solution of the main
  equation~\eqref{basic} for $g_{s'}$.
We define the linear map $\phi_{s'}:Sol(g_s)\to Sol(g_{s'})$ by 
  $$\phi_{s'}(a)=\left. a\right|_{s'}$$

  If two solutions $a_{ij}$ and $a'_{ij}$ coincide on an open subset,
  they coincide everywhere. Thus, for every $s'$ we have $\ker
  \phi_{s'} = 0$, so we obtain
  $$\dim Sol(g_{s'})\geq \dim Sol(g_s)\geq \dim Sol(g).$$
  
  Then $D(g_1),D(g_2),...,D(g_s),...$ is a  semidecreasing (in the sence $D(g_s)\ge D(s')$ for $s<s'$)  sequence of
  natural  numbers. Therefore, there exists a number $k$ such that
  $D(g_k)=D(g_{k'})$ for all $k'\geq k$. As we explained above,
  $D(g)\leq D(g_k)$. Let us show that $D(g)\geq D(g_k)$.

  Consider an arbitrary $k'\geq k$. Then $\phi_k (Sol(g_{k'}))\subset
  Sol(g_k)$ and $\dim Sol(g_{k'}) = \dim Sol(g_k)$. Since $\phi_k$ is
  a linear map with zero kernel, we have $\phi_k (Sol(g_{k'})) =
  Sol(g_k)$. Thus, every solution $a\in Sol(g_k)$ on $M_k$ can be
  uniquely extended to the solution $a\in Sol(g_{k'})$ on $M_{k'}$.

  Now we consider $\phi_k: Sol(g)\to Sol(g_k)$. Our goal is to show
  that $\phi_k( Sol(g)) = Sol(g_k)$. We choose an arbitrary $a\in
  Sol(g_k)$ on $M_k$ and define its extension $A\in Sol(g)$ on $M$ in
  the following way: For every point $P\in M$ there exists $k'\geq k$
  such that  some neighborhood of $P$ lies in $M_{k'}$. Then there exists
  extension $a'\in Sol(g_{k'})$ of $a$, such that $\phi_k a' =a$. We
  define $A(P):=a'(P)$. Clearly, this construction does not depend on
  the choice of $k'$, so $A(P)$ is well-defined for all $P \in M$.
  By construction it satisfies~\eqref{basic} on $M$. Then, $A\in
  Sol(g)$ and $\phi_k (A)=a \in Sol(g_k)$.

  We obtain that $\phi_k(Sol(g))=Sol(g_k)$ and, therefore, $\dim
  Sol(g_k)=\dim Sol(g)$.
\end{proof}

Combining Lemma~\ref{thm:mu_neq_0} and 
Lemma~\ref{thm:noncompact}, we obtain Theorem \ref{th:main} under the additional assumption that $B=0$ and 
$\mu\ne 0$. Indeed, take a  sequence $M_s$ of 
  simply-connected connected  open subsets of $M$ 
    such that 
      $M_s\subset M_{s+1}$, 
   each $M_s$ has compact closure, and 
     $\bigcup_{s  = 1}^\infty M_s = M$.   Then, by
  Lemma~\ref{thm:noncompact}, there exists $k$ such that 
 the degree of monbility of    $ g_k=g|_{M_k}$  is $D(g)$. 
 By  Lemma~\ref{thm:mu_neq_0},  there exists $\bar
  g_k$ on $M_k$ which is geodesically equivalent to $g_k$ on $M_k$,
  with $\bar B\neq 0$.
  Then, $D(g)=D(g_k)=D(\bar g_k).$
  
  Since $\bar g_k$ satisfies the conditions of 
  Theorem~\ref{th:main} and has $\bar B= B(\bar g_k)\ne 0$, we have 
     $D(g)=D(\bar g_k)=\frac{k(k+1)}{2} + \ell$ for a
  certain $k\in \{0,1,...,n-2\}$ and $1 \le \ell \le \lfloor
  \tfrac{n-k+1}{3}\rfloor$. Theorem \ref{th:main}
  is proved under the assumption that $B=0$ but there exists a solution $(a, \lambda, \mu)$ with $\mu\ne 0$.

\section{ Proof of Theorem \ref{th:main} if all solutions $(a, \lambda, \mu)$ have $\mu =0$. }
\label{sec:B=0_isotropic}
In fact, we show that in this case the list of degrees of mobilities of $g$ is smaller than in the generic case $B\ne 0$:

\begin{lem}
  \label{thm:B=0_isotropic}
  Let $g$ be a Lorentzian metric on a connected simply-connected
  manifold $M$   admitting a metric $\bar g$ that 
   is geodesically equivalent but not affinely equivalent to $g$. Suppose that
  $D(g)\geq 3$, the corresponding constant $B$ is equal to 0,  and that every solution  $(a, \lambda, \mu)$ 
  of~\eqref{eq:mg+Ba} has  $\mu=0$.
  
  Then,
  $D(g)=\tfrac{k(k+1)}{2}+\ell,$
  where $1\leq k\leq n-3$ and $2\leq \ell\leq
  \lfloor\frac{n-k-1}{3}\rfloor$.
\end{lem}

\subsection{ Technical statements that will be used in proof of Lemma \ref{thm:B=0_isotropic}. } 
Within this section we assume that $(M,g)$ is a connected  simply connected $n\ge 3$-dimensional  manifold of riemannian or lorentzian signature  with $D(g)\ge 3$ and $B=0$. 

\begin{lem}
  \label{thm:v_perp_l}
  Assume  all solutions
  of the extended system~\eqref{eq:mg} have $\mu=0$. Let
  $(a_{ij},\lambda_i,0)$ be an arbitrary solution.

  Then, $\lambda_i$ is parallel and orthogonal to any
  parallel  1-form on $M$. In particular,  if $\lambda_i\ne 0$, then it 
  is isotropic and the signature of $g$ is lorentzian.
\end{lem}
\begin{proof}
  
 Since $B=0$ and $\mu=0$, the extended system~\eqref{eq:mg+Ba} reads 
  \begin{gather}
    \label{eq:l_ij=0} 
    \left\{
      \begin{array}{ll} 
        a_{i j,k}&=\lambda_i g_{j k} + \lambda_j g_{i
          k}\\
        \lambda_{i,j}&=0.
      \end{array} \right.
  \end{gather}

  Thus, for every solution $(a_{ij},\lambda_i,0)$ of the extended
  system, $\lambda_i$ is parallel as we claimed. 
  As we explained in  Section~\ref{standard}, $\lambda_i=\lambda_{,i}$ for the function  $\lambda:=\frac{1}{2} \tr_g a.$
 
  Consider an arbitrary parallel 1-form $v_i$ on
  $M$. It is evidently closed; since our manifold is   
  simply-connected,  there exists a function $v$ such that  $v_i= v_{,i}$. We
  take the  1-form  $u_i$ given by 
  \begin{gather}
    \label{eq:sol_u} u_i=a_{ij} v^j - v \lambda_i.
  \end{gather}

We have 
  \begin{gather}
    \label{eq:u_ij}
    u_{i,k}=a_{ij,k} v^j + a_{ij} v^j_{,k} - v_k \lambda_i=\lambda_i
    g_{jk} v^j +\lambda_j g_{ik} v^j - v_k \lambda_i = (\lambda_j v^j)
    g_{ik}
  \end{gather}

  Let us now take $a'_{ij}=u_i u_j$ and show that $a'$ is a solution
  of~\eqref{eq:mg}.
  Indeed,  
  \begin{gather}
    a'_{ij,k}=u_i u_{j,k} + u_{i,k} u_j= u_i \lambda_q
    v^q g_{jk}+ u_j \lambda_q v^q g_{ik}.
  \end{gather}
  Thus, $a'_{ij}$ satisfies the first equation of \eqref{eq:mg} with $\lambda'_i=\lambda_q
  v^q u_i$. In order to  calculate the corresponding  $\mu'$ we use  the second equation
  of~\eqref{eq:mg}:
  $$\lambda'_{i,j}=\lambda_q v^q u_{i,j}\overset{\eqref{eq:u_ij}}
  {=}(\lambda_q v^q)^2 g_{ij}.$$ We see that  $\mu'=(\lambda_q v^q)^2$.
Thus,   for 
  parallel  $v_i$ we have constructed the new
  solution $(a':=u_i u_j,\lambda':=\lambda_q v^q
    u_i,\mu':=(\lambda_q v^q)^2)$ of~\eqref{eq:mg}.  
  By assumption every solution of \eqref{eq:mg}  has  $\mu=0$ implying $\lambda_i$ is orthogonal
  to $v_i$ as we claimed.
\end{proof}

\begin{lem}
\label{thm:mu=0_jordan}
  Let $g$ be a Lorentzian metric such that  $B=0$ and such   that all solutions
  of the extended system~\eqref{eq:mg} have $\mu=0$ and let
  $(a_{ij},\lambda_i,0)$ be an arbitrary solution with $\lambda_i\neq
  0$.

  Then,  there exists a constant $C$ such that (at every point $p\in M$) \ {}  
  $\lambda^i$ is an eigenvector of $a^i_j$ with eigenvalue
  $\lambda+C$.
  Moreover, all other  eigenvalues of $a^i_j$  are constants.
   In a generic point the eigenvalue $\lambda+C$  has algebraic multiplicity $2$, geometric multiplicity $1$,    and corresponds to the $2$-dimensional nontrivial Jordan block of
  $a^i_j$.

  In other words, in a generic point of $M$ 
  the Jordan form of $a^i_j$ looks as follows:
  \begin{gather}
    \label{eq:mu=0_jordan}
    a^i_j=\left(\begin{array}{cc|ccc|c|ccc}
        \lambda+C  &1&&&&&&&\\
        &\lambda+C &&&&&&&\\
        \hline
        &&\rho_2 &&&&&\\
        &&&\ddots &&&&&\\
        &&&&\rho_2 &&&&\\
        \hline
        &&&&&\ddots
        &&\\
        \hline
        &&&&&&\rho_m &&\\
        &&&&&&&\ddots &\\
        &&&&&&&&\rho_m
      \end{array} \right)
  \end{gather}
  where $\lambda=\frac{1}{2}\tr_g a$,  where $\rho_2,\dots , \rho_m$ are 
  constant eigenvalues of multiplicities  $k_2,\dots, k_m$ respectively and  $C:=-\frac{1}{2}\sum_{s\geq 2}^m k_s \rho_s$.
\end{lem}

\begin{proof}
  In order to show that $\lambda_i$ is the eigenvector of $a_{ij}$, 
  we   construct $u_i$ as
  in~\eqref{eq:sol_u}  with $\lambda_i$ playing the role of $v_i$: 
  $$u_i=a_{ij} \lambda^j - \lambda \lambda_i.$$

  Then, $u_{i,k}=(\lambda_j\lambda^j) g_{ik}=0$ and $u_i$ is a
  parallel  1-form on $M$. By  Lemma \ref{thm:v_perp_l}, it is orthogonal to
  $\lambda_i$, so we have  $$0=u_i\lambda^i=(a_{ij}\lambda^j-\lambda
 \lambda_i)\lambda^i=a_{ij}\lambda^j \lambda^i.$$
  
Then, there exists a function $u$ such that $u_{, i}  =u_i$.  
 Next,  define $U_i=a_{ij} u^j - u \lambda_i$
 (similar to~\eqref{eq:sol_u}). By direct calculations we  see  $U_{i,k}=(\lambda_j u^j) g_{ik}=0.$
 Thus, $U_i$ is parallel and in view of Lemma \ref{thm:mu_neq_0} orthogonal to
 $\lambda_i$. Hence,
 $$0=U_i\lambda^i=(a_{ij}V^j-V \lambda_i)\lambda^i=a_{ij}V^j
 \lambda^i=a_{ij} u^j \lambda^i=a_{ij}(a^{j}_l \lambda^l-\lambda
 \lambda^j) \lambda^i=a_{ij} a^j_l \lambda^l \lambda^i.$$

 At every point $P$, consider  $S=\textrm{span} \{\lambda_i,u_i\} \subset T_P^*M$. We have 
 $ \lambda_i
 u^i=\lambda_i
 \lambda^i  =0 $. Moreover, since $a_{ij}\lambda^i \lambda^j=0$ and
 $a_{ij}a^j_l \lambda^l\lambda^i=0$, we also have  $u_i u^i=0$. Indeed, 
 $$u_i u^i=(a_{ij} \lambda^j - \lambda \lambda_i)(a^i_l \lambda^l
 - \lambda \lambda^i)=  a_{ij}\lambda^j a^i_l \lambda^l=0.$$
  
 Therefore, $S$ is a totally isotropic subspace. Since $g$
 is Lorentzian, the dimension of $S$ is at most $1$. Thus, the 1-forms  $u_i$ and $\lambda_i$ are linearly dependent everywhere on
 $M$. Since they are   parallel and $\lambda_i\ne 0$, there exists
 a constant $C$ such that $u_i=C \lambda_i$. Then,
 $$a^i_j \lambda^j=u^i + \lambda \lambda^i= (C+\lambda) \lambda^i, $$
 i.e.   $\lambda^i$ is an eigenvector of 
 $a^i_j$ whose  eigenvalue  is $(\lambda+C)$ as we claimed.

 Next  we calculate the algebraic  multiplicity of the eigenvalue.
 We assume that we work at a  point  such that the multiplicities of
 the eigenvalues of $a^i_j$ are the same in a small neighborhood; almost every point has this property. 
 Near such points, the eigenvalues $\rho_1=\lambda+C,\rho_2,\dots,\rho_m$ are
 well-defined smooth functions.

By  Splitting Lemma (\cite[Theorem 3]{Bols}; actually at this point we need only \cite[Theorem 1]{Bols}), there exists a local  coordinate system
 $(x^{(1)}_1,\dots,x^{(1)}_{k_1},\dots,x^{(m)}_1,\dots,x^{(m)}_{k_m})$,
 such that each eigenvalue $\rho_i$ depends only on the coordinates
 $(x^{(i)}_1,\dots,x^{(i)}_{k_i})$. Clearly, 
 \begin{multline} \label{-12} 
   \tr a^i_j =2 \lambda(x^{(1)}_1,\dots,x^{(1)}_{k_1})
   =\\ =k_1(\lambda(x^{(1)}_1,\dots,x^{(1)}_{k_1})+C)+k_2
   \rho_2(x^{(2)}_1,\dots,x^{(2)}_{k_2})+\dots +k_m
   \rho_m(x^{(m)}_1,\dots,x^{(m)}_{k_m}).
 \end{multline}
  Differentiating    \eqref{-12}   with respect to  $x_i=x^{(1)}_i$   we
 obtain $2 \lambda_i=k_1 \lambda_i$. 
 Since $\lambda_i\neq 0$ we have $k_1=2$  as we claimed.
 Differentiating \eqref{-12}   with respect to  $x^{(s)}_j$ with $s>1$ we
 obtain $\frac{\partial}{\partial x^{(s)}_j} \rho_s=0.$
 Thus, all $\rho_s$ for $s\geq 2$ are constants  as we claimed.

 Our  next goal is to show that the $2$-dimensional Jordan block corresponding to
 the eigenvalue $\lambda+C$ is nontrivial (i.e., is not proportional to $\Id$).

 By Splitting Lemma, the distribution of the
 generalized eigenspaces is integrable. Thus, locally there exists
 a 2-dimensional submanifold $N$ in $M$ whose  tangent space is 
 $TN=\ker \left((a^i_j-(\lambda+C)\Id)^2\right)$.

 Moreover, by Splitting Lemma,  there exists a metric $h$ to $N$,
 such that the restriction of $a^i_j$ to  $N$ is a solution of the
  equations~\eqref{basic} for the metric $h$ on
 $N$, and such that  its only eigenvalue is $(\lambda+C)$.
If the restriction  of $a^i_j$ to $TN$ is diagonal, it must be  $a^i_j=(\lambda+C) \delta^i_j$, so the 
  tensor field $a_{ij}=(\lambda+C) h_{ij}$. By \cite[Lemma 4]{KM}, $\lambda+C$ is constant on $N$ implying it is constant on $M$ which contradicts the assumptions.  

 Then, the restriction of $a^i_j$ to its generalized eigenspace $TN$
 is not diagonal, therefore, is similar to the
 nontrivial $2$-dimensional Jordan block.
  
 Since $g$ has the  lorentzian signature,  a selfadjoint  endomorphism 
 does not admit more than one nontrivial Jordan block by Theorem \ref{thm:pairofmatrices}. Therefore,
 $a^i_j$ has the Jordan form~\eqref{eq:mu=0_jordan}.
\end{proof}

Next we consider the metric admitting the solution $a^i_j$ whose Jordan form (at almost every point) is~\eqref{eq:mu=0_jordan} and show that if $(M,g)$ is
indecomposable and does not admit solutions with $\mu\neq 0$, then
$a^i_j$ has at least $2$ different constant eigenvalues and all
constant eigenvalues of $a^i_j$ have  multiplicities at least $2$.

\begin{lem}
  \label{thm:rho_1,rho_2} 
  Suppose that almost everywhere on $M$ the  tensor field $a^i_j$ has the 
  Jordan form~\eqref{eq:mu=0_jordan}. Assume $(a_{ij}, \lambda_i)$ is a solution of \eqref{basic} such that $\lambda_i$ is parallel and isotropic. 
  Then,  the following statements hold: 
  \begin{enumerate}
  \item If $m=2$, i.e. $a^i_j$ has only one constant eigenvalue, there
    exists a  1-form   $v_i$ satisfying $v_{i,j}= g_{ij}$. 
  \item If, for a certain $s>1$,  the eigenvalue $\rho_s$ has multiplicity
    $k_s=1$, then there exists a  parallel 
    1-form  on $M$ that is  linearly independent of  $\lambda_i$.
  \end{enumerate}
\end{lem}
\begin{proof} 
  We first  describe  $g$ in a
  neighborhood of almost every point in $M$.
  We  denote the characteristic polynomial of $a^i_j$ by
  $\chi(t)$ and consider its decomposition into coprime components
  $\chi_s(t)=(t-\rho_s)^{k_s}$:
  $$\chi(t)=\underbrace{(t-\lambda-C)^2}_{\chi_1(t)}\cdot 
  \underbrace{(t-\rho_2)^{k_2}}_{\chi_2(t)}\cdot\dotsb\cdot
  \underbrace{(t-\rho_m)^{k_m}}_{\chi_m(t)}.$$

  This decomposition is ``admissible'' in the terminology  of  Splitting
  Lemma. Therefore, there exists a coordinate system
  $(x_1,x_2,x^{(2)}_1,\dots,x^{(2)}_{k_2},\dots,x^{(m)}_1,
  \dots,x^{(m)}_{k_m})$ on $M$ such that $a^i_j$ and $g_{ij}$ have the
  following block-diagonal form:

  \begin{gather}
    a^i_j=\left(\begin{array}{cc|ccc|c|ccc}
        \lambda(x_1,x_2)+C  &f(x_1,x_2)&&&&&&&\\
        &\lambda(x_1,x_2)+C &&&&&&&\\
        \hline
        && &&&&&\\
        &&&\rho_2 \Id_{k_2} &&&&&\\
        &&&& &&&&\\
        \hline &&&&&\ddots
        &&\\
        \hline
        &&&&&& &&\\
        &&&&&&& \rho_m \Id_{k_m} &\\
        &&&&&&&&
      \end{array} \right),\\
    g_{ij}=\left(
      \begin{array}{cc|c|c|c}
        \multicolumn{2}{c|}{ \multirow{2}{*}{$h_1(x_1,x_2)\cdot
            \hat\chi_1(A_1)$}} &&&\\
        &&&\\
        \hline
        && h_2(x^{(2)})\cdot\hat\chi_2(A_2) &&\\
        \hline
        &&&\ddots&\\
        \hline
        &&&& h_m(x^{(m)})\cdot \hat\chi_m(A_m) 
      \end{array}\right)
  \end{gather}
  where $\Id_k$ is an $k$-dimensional identity endomorphism, $A_s$ are  $k_s\times k_s$ matrices
given by 
  $$A_1=\left(\begin{array}{cc}
      \lambda(x_1,x_2)+C  &f(x_1,x_2)\\
      &\lambda(x_1,x_2)+C
    \end{array}
  \right)\quad\mbox{and}\quad A_s=\rho_s\Id_{k_s},$$  $h_s$ are  nondegenerate  symmetric     matrices such that the entries of $h_s$   depend  only on the coordinates
  $x^{(s)}_1,\dots,x^{(s)}_{k_s}$,  and such that $h_s$ is positively definite  for $s\ge 2$,  and $\hat\chi_s(t):=\frac{\chi(t)}{\chi_s(t)}$ (it is a polynomial of degree $n-k_s$).
    
  Note that, since for all $s\geq 2$ the eigenvalues $\rho_s$ are
  constant, $\hat\chi_1(t)$ has constant coefficients. Thus,
  $\hat\chi_1(A_1)$ depends on the variables $(x_1,x_2)$ only. 
  Since $A_s=\rho_s \Id_{k_s}$,
  $$\hat\chi_s(A_s)=\hat\chi_s(\rho_s \Id_{k_s})
  =\hat\chi_s(\rho_s)\Id_{k_s}=\const \cdot (\rho_s-\lambda-C)^2 \Id_{k_s}. $$

  Therefore we can rewrite metric $g$ in the following form:
  \begin{gather}
    \label{eq:block-diagG}
    g_{ij}=\left(
      \begin{array}{cc|c|c|c}
        \multicolumn{2}{c|}{ \multirow{2}{*}{$g_1(x_1,x_2)$}} &&&\\
        &&&\\
        \hline
        && (\lambda(x_1,x_2)+C-\rho_2)^2 g_2(x^{(2)}) &&\\
        \hline
        &&&\ddots&\\
        \hline
        &&&& (\lambda(x_1,x_2)+C-\rho_m)^2 g_m(x^{(m)})
      \end{array}\right)
  \end{gather}
  where $g_s$ for $s\ge 2$ are certain  positively  defined symmetric $k_s\times k_s$-matrices  depending only on the coordinates
  $x^{(s)}_1,\dots,x^{(s)}_{k_s}$.

Let us now prove the Lemma under the assumption of   Case $(1)$: we assume  $m=2$ so  $a^i_j$ has only one constant eigenvalue   $\rho_2$ of multiplicity $k_2$.
Instead  we consider $a_{ij} - C g_{ij}$; the pair $(a_{ij} - C g_{ij}, \lambda_i)$ is clearly a solution of \eqref{basic}.
 Clearly, 
  $\lambda^i$ is the eigenvector of $a^i_{j} - C \delta^i_{j}$ with eigenvalue $\lambda$ and all other eigenvalues of $L_j^i$ are zero. Then,  the matrix of 
     $a^i_{j} - C \delta^i_{j}$ in our coordinate system is  given by 
  \begin{gather}
   a^i_{j} - C \delta^i_{j}=\left(\begin{array}{cc|ccc}
        \lambda(x_1,x_2)  &f(x_1,x_2)&&\\
        &\lambda(x_1,x_2) &&&\\
        \hline
        && &&\\
        &&& \textrm{\Large 0}  &\\
        &&&&\\
        \end{array} \right).
  \end{gather}

  We consider  the  (unique) 1-form  $v_i$  such that 
  \begin{equation}\label{ma}  
a_{ij} - C g_{ij}= v_i \lambda_j +  v_j \lambda_i. 
  \end{equation}
  Such 1-form   exists at almost every point   since  rank of $a_{ij} - C g_{ij}$ is two and since 
   $\lambda_i\ne 0$ in the image of $L^{i}_j$. 
 In order to show the existence everywhere, we observe that 
  \eqref{ma} is a system of linear equations on the components of $v_i$ whose coefficients (i.e. the components of $\lambda_i$ and of $L_{ij}$) smoothly depend on the positions. Then, 
  the existence of a solution almost everywhere implies the existence of a solution everywhere. The uniqueness  of the solution 
  follows from $\lambda_i\ne 0$ (which is fulfilled everywhere since $\lambda_i$ is parallel) 
  and implies that   $v_i$ is smooth. 
  
  Covariantly differentiating \eqref{ma} and using \eqref{basic}, 
  we obtain 
 $$
\lambda_i g_{jk} + \lambda_j g_{ik} =    \lambda_jv_{i,k} + \lambda_i v_{j,k}  
 $$ 
  implying $\lambda_i (g_{jk} - v_{j,k}) +  \lambda_j (g_{ik} - v_{i,k})=0$ implying $g_{ik}= v_{i,k}$ as we want. 
   Lemma is proved under the  assumptions of Case (1). 

In order to prove the Lemma under the assumptions of    Case $(2)$, we  suppose that $a^i_j$ (in a generic point)  has a 
constant eigenvalue of
  multiplicity $1$. We renumerate the eigenvalues such that the last
  eigenvalue $\rho_m$ has multiplicity $k_m=1$.

  Then, the last component of the metric $g$ in~\eqref{eq:block-diagG}
  is one-dimensional, and one can choose (locally, in a neighborhood of a generic point) a coordinate
  $w=x_n$  such that the corresponding 1-form
  $w_k=w_{,k}$ satisfies the following conditions:
  \begin{eqnarray}
    \label{eq:a*w}
    a^i_jw^j&=&\rho_m w^j\\
    w_i w^i&=&\frac{1}{(\lambda+C-\rho_m)^2}
  \end{eqnarray}
  We consider  the following 1-form 
  \begin{gather}
    \label{eq:constu}
    u_i=(\lambda+C-\rho_m)w_i- w\lambda_i
  \end{gather}
  and show that it is parallel.
  First we describe  how  the (1,1)-tensor $w^j_{\ ,k}$ viewed as an endomorphism  acts on the basis vectors of the
  tangent space $T_p M$.   Note that $\lambda_i w^i=0$, since both vectors are eigenvectors of
  $a^i_j$ with different eigenvalues.
  We covarinatly  differentiate~\eqref{eq:a*w}  and
  substitute $a_{ij,k}=\lambda_i g_{jk}+\lambda_j g_{ik}$ to obtain:
  \begin{gather}
    \label{eq:a*w_k}
    \lambda_i w_k+a_{ij}w^j_{\ ,k}=\rho_m w_{i ,k}.
  \end{gather}
  
  Now we contract the equation~\eqref{eq:a*w_k} with an arbitrary
  eigenvector $\tau^i$, such that $a^i_j\tau^j=T \tau^j$. We obtain 
  $(\lambda_i\tau^i) w_k+T \tau_j w^j_{\ ,k}=\rho_m \tau^i w_{i,k}. $
  Thus, $(T-\rho_m) w^j_{\ ,k}\tau_j=\lambda_i \tau^i w_k.$

  For all basis eigenvectors $\tau^i$ with eigenvalues
  $\rho_1=\lambda+C, \rho_2,\dots \rho_{m-1}$ we have $T\neq \rho_m$
  and $\lambda_i\tau^i=0$. Therefore, for every such vector  we have $w^j_{\
    ,k}\tau_j=0$.

  Let us put $\tau_i$ equal to the last eigenvector $w_i$ and
  calculate $w^j_{\ ,k} w_j$:
  $$w^j_{\ ,k} w_j=\frac{1}{2}(w_j w^j)_{,k}=\partial_k
  \left(\frac{1}{2(\lambda+C -\rho_m)^2}\right)=-\frac{1}{(\lambda+ C-\rho_m)^3}
  \lambda_k.$$

  As  the remaining basis vector of $T_p
  M$ we take  $v^i$  such that  $a^i_j
  v^j=(\lambda+C) v^i+\lambda^i$.

  Then, contracting \eqref{eq:a*w_k} with $v^i$ we have:
  $  (\lambda_i v^i) w_k+((\lambda+C) v_j+\lambda_j)w^j_{\ ,k}=\rho_m w_{i ,k} v^i.$
  Thus,
  $w^j_{\ ,k} v_j=-\frac{\lambda_i v^i}{(\lambda+C-\rho_m)}w_k.$
  
  We have constructed the basis of $T_p M$ whose first  $(n-1)$ vectors are 
  eigenvectors of $a^i_j$ and the last vector  is the vector  $v_i$, and
  calculated entries of endomorphism  $w^j_{\ ,k}$ in this basis. We substitute it in   the derivative of ~\eqref{eq:constu} which is 
  $u_{i,j}=(\lambda+C-\rho_m)w_{i,j}+\lambda_j w_i + w_j \lambda_i.$ 
In order to show $u_{i,j}= 0$ it is sufficient to show that we obtain zero 1-form if we  contract $u_{i,j}$ with all vectors of our basis.   

 For any eigenvector $\tau^i$ of $a^i_j$ corresponding to the
  eigenvalues $\rho_1=\lambda+C,\rho_2,\dots, \rho_{m-1},$ we have
  $u_{i,j}\tau^j=(\lambda+C-\rho_m)w_{i,j}\tau^j+\lambda_j\tau^j w_i
  + w_j\tau^j \lambda_i=0, $
  because  $\tau_i$ is orthogonal to both $\lambda_i$ and $w_i$. 
  Moreover,    $u_{i,j}w^j=(\lambda+C-\rho_m)w_{i,j}w^j+  \lambda_j w^j w_i + w_j \lambda_i w^j=0.$
  For the remaining  basis vector $v^i$ we have 
  $u_{i,j}v^j=(\lambda+C-\rho_m)w_{i,j}v^j+\lambda_j v^j w_i +
  w_j \lambda_i v^j=0.$
  Therefore, $u_{i,j}=0$ and $u_i$ is a  (nonzero) parallel 1-form  linearly independent of  $\lambda_i$, whose existence we claimed.

We constructed the 1-form $u_i$  at generic point only. 
 In order  to extend $u_i$ to the whole manifold $M$, we consider the
  distribution $W$ on $M$ defined as follows:
  
  In a neighborhood of a  point $P$ 
  such that $\lambda(P)+C\ne \rho_s$ for $s=2,...,m$   we put
  $$W=\ker \left((a-(\lambda+C)\Id)(a-\rho_m \Id)\right)=span\{u_i,\lambda_i\}.$$

  In a neighborhood of  a  point $P$ 
  such that $\lambda(P)+C= \rho_s$ for  some $s=2,...,m-1$, \    we put
  $$ W= \{\const\cdot \lambda^i\}\oplus \ker (a-\rho_m \Id). $$
  
  And in a neighborhood of a  point $P$ such that 
  $\lambda(P)+C=\rho_m$ we put 
  $$W=\ker (a-(\lambda+C)\Id)(a-\rho_m \Id)\cap
  \lambda_i^\perp.$$

  It is easy to see that $W$ is a well-defined smooth $2$-dimensional
  distribution on $M$; moreover,  almost everywhere it  coincides with a
  linear span of two parallel   vector fields $\lambda^i$
  and $u^i$. Thus, it is parallel and flat almost everywhere and,
  therefore, everywhere on $M$. Then, there exist a  globally defined parallel  1-form  on $M$ that is 
  linearly independent of $\lambda_i$.
\end{proof}

\begin{cor}
  \label{thm:lorentz_dim}
  Let $(M,g)$ be an indecomposable Lorentzian manifold with $D(g)\geq
  3$, such that $B=0$ and all solutions of the extended
  system~\eqref{eq:mg} have $\mu=0$.
  Suppose there exists at least one solution $(a_{ij},\lambda_i,0)$
  of~\eqref{eq:mg+Ba} such that $\lambda_i\neq 0$.
  Then, the dimension of $M$ is at least $6$.
\end{cor}
\begin{proof}
  By Lemma~\ref{thm:mu=0_jordan}, in a certain basis   the matrix of $a^i_j$ has the form~\eqref{eq:mu=0_jordan}.
Since $M$ does not admit solutions of \eqref{eq:mg+Ba} such that $\mu\ne 0$,  there exists no $v_i$ such that $v_{i,j}=0$. Indeed, for such $v_i$ the triple $(a'_{ij}= v_iv_j, \lambda'_i= v_i, \mu'_i= 1)$ is a solution of \eqref{eq:mg}.  Then, by Lemma~\ref{thm:rho_1,rho_2},
$m\ge 3.$ 
  Since $M$ is Lorentzian and indecomposable, it does not admit 
  parallel  1-forms that are not   constant multiples of  $\lambda_i$. Then, by Lemma~\ref{thm:rho_1,rho_2}, $k\ge 3$. 
  Thus,   $\dim M\geq 2+2+2=6$ as we claimed.
\end{proof}

\begin{cor}  \label{cor3} Assume $g$ (on a connected simply connected manifold)
 has lorentzian signature, $D(g)=3$, $B=0$ and every solution $(a, \lambda, \mu)$  of the extended system~\eqref{eq:mg} has  $\mu=0$. Then, every homothety vector field of $g$ is an isometry. 
\end{cor} 
\begin{proof} We will work in a neighborhood of a generic point and  consider the coordinates  $$(x_1,x_2,x^{(2)}_1,\dots,x^{(2)}_{k_2},\dots,x^{(m)}_1,\dots,x^{(m)}_{k_m})$$  as above such that $g$ has the form  \eqref{eq:block-diagG}. By Lemma \ref{thm:rho_1,rho_2}, $m\ge 2$. 

  Clearly, any homothety (and therefore any Killing) vector field has the form 
\begin{equation}\label{(2)}v^i = \left(v_1(x_1, x_2), v_2(x_1, x_2), v_1^{(2)}(x^{(2)}),\dots, v_{k_2}^{(2)}(x^{(2)}),\dots,   v_1^{(m)}(x^{(m)}),\dots, v_{k_m}^{(m)}(x^{(m)})\right).\end{equation}

Indeed, any homothety sends $g$ to $\const \cdot g$ for $\const \not= 0$ and the solution $a_{ij}$ to a nontrivial solution, that is to a tensor of the form $C_1 a_{ij} +C_2g_{ij} +C_3\lambda_i\lambda_j$ (for $C_1 \not= 0$). The pair
$(\const\cdot  g,C_1 a + C_2 g + C_ 3\lambda\otimes   \lambda )$
 determines the foliations corresponding to the coordinate plaques
$(x_1, x_2)$,  $x^{(2)}$,\dots,  $x^{(m)}$ uniquely, since  the foliations do not depend on the choice of constants $const \not = 0$, $C_1 \not= 0$, $C_2$ and  $C_3$. 
Then, any homothety preserves the foliations and therefore  has the form \eqref{(2)}.

We take any $s=2,...,m$ and consider the coordinate  plaque of the coordinates 
$(x^{(s)}_{1},\dots, x^{(s)}_{k_s})$, i.e., 
 the $k_s$-dimensional submanifold given by the equations 
 $$x_1= \const_1, x_2= \const_2,\dots, x^{(s-1)}_{k_{s-1}}= \const^{(s-1)}_{k_{s-1}}, x^{(s+1)}_1 =\const^{(s+1)}_1,\dots, x^{(m)}_{k_m} =\const^{(m)}_{k_m}$$ with the restriction of the metric $g$ to it which in view of  \eqref{eq:block-diagG} has the form 
$(\lambda+ C - \rho_s)^2 g_s$, and the orthogonal projection of $v^i $ to this plaque which has the form 
$v^i_{restr}:= \left( v_1^{(s)}(x^{(s)}),\dots, v_{k_s}^{(s)}(x^{(s)})\right).$ Since the vector field $v^i$ is homothety, the vector field $v^i_{restr}$ is also a homothety (possibly, with another coefficient) so the 
pullback $\phi_t^*g_s$ w.r.t. to the flow of the vector field is equal to $ \exp(\alpha_s t) g_s.$
For another $s$ (which we denote by $s'$), by repeating  the arguments, 
  we also   have $\phi_t^*g_{s'}= \exp(\alpha_{s'} t) g_{s'}.$ Now, since  the homothety  sends the (unique, up to a factor) 
   covariantly constant 1-form $\lambda_i$   to $\beta\cdot \lambda_i$, we have that the evolution of the function $\lambda$ along a  trajectory of the flow of $v^i$ is given by $\lambda(t):= \lambda(\phi_t(p))= \beta \lambda(p) + \gamma$.    All together, we obtain  
   \begin{equation} \label{tttt}
   \phi_t^*g_{s'}= \exp(\alpha_s t) g_{s'} , \ \phi_t^*g_{s'}= \exp(\alpha_{s'} t) g_{s'}, \ \lambda(t)= \beta \lambda(p) + \gamma.\end{equation}  Combining this with the assumption that 
    the flow  of $v^i$ acts by homotheties, we obtain 
    $$(\beta \lambda(p) + \gamma - \rho_s)^2 \exp(\alpha_s t) 
    = (\beta \lambda(p) + \gamma - \rho_{s'})^2 \exp(\alpha_{s'} t). 
   $$
   Then, $ \alpha_s= \alpha_{s'}$ and  $\beta= 0$ which implies that the vector field $v^i$ is a Killing vector field.  
\end{proof}
\subsection{Proof of  Lemma~\ref{thm:B=0_isotropic}.}

  Since  $g$ admits at least one metric which is geodesically
  equivalent, but not affinely equivalent to $g$, there exists at
  least one solution $(a_{ij},\lambda_i)$ with nonzero vector field
  $\lambda_i$.

  Let us first show that if $a_{ij}$ and $\hat a_{ij}$ are solutions
  of~\eqref{eq:mg} with nonzero vector fields $\lambda_i$ and $\hat
  \lambda_i$ respectively, then there exists a constant $C$, such that
  $C a_{ij}- \hat a_{ij}$ is parallel.

  We consider the space $S=span \{\lambda_i,\hat \lambda_i\}$. By
  Lemma~\ref{thm:v_perp_l}, $S$ is totally isotropic. Since $g$ has
  lorentzian signature, $\dim S$ is at most 1. Thus, there exists $C$
  such that $C \lambda_i=\hat \lambda_i$. Since both vector fields are
  parallel  on $M$, $C$ is constant. Then,
  $$(C a_{ij}-\hat a_{ij})_{,k}=(C \lambda_i-\hat \lambda_i)g_{jk}
  +(C \lambda_j-\hat \lambda_j)g_{ik}=0$$
 so   $C a_{ij}-\hat a_{ij}$ is parallel. 
  
  Thus, the space of solutions of the extended system~\eqref{eq:mg+Ba}
  is  the direct sum of the space $Par(g)$ of parallel  symmetric $(0,2)$-tensor fields and one-dimensional space
  $\{C\cdot a_{ij}\}$. Then, $D(g)=\dim Par(g)+1.$

  In order to calculate  $\dim Par(g)$ we will use essentially the
  same construction as in Theorem~\ref{cov1}.

  Consider the decomposition of a tangent space $T_p M$ into the
  direct sum of subspaces, invariant with respect to the action of the
  holonomy group $Hol_p(M)$ 
  \begin{gather}
    \label{eq:B=0-decomp}
    T_p M=V_0\oplus V_1\oplus\dots \oplus V_\ell
  \end{gather}
  where $V_0$ is maximal nondegenerate flat subspace and $V_s$, $s>0$, are
  indecomposable nondegenerate subspaces. We denote the restriction of
  $g$ to $V_s$ by $g_s$.

  All parallel symmetric  tensor fields on the Lorentzian
  manifold $M$ are given by the formula
  \begin{equation}
    A= \sum_{i,j=1}^k c_{ij} \tau_i\otimes\tau_j+
    \sum_{i=1}^\ell C_i g_i,
  \end{equation} 
  where $c_{ij}$ is a constant symmetric matrix, $C_1,...,C_\ell$ are
  constants and $\tau_i$ are the basis in the space of all parallel  $1$-forms on $M$.
  Then, 
  $\dim Par(g)=\frac{k(k+1)}{2}+\ell. $

  In order to  complete  the proof we need to show that $\ell$ and $k$ satisfy 
  $1\leq k\leq n-3$ and $1\leq \ell\leq
  \lfloor\frac{n-k}{3}\rfloor$.

  Recall that in the case of cone manifold, the estimation uses the existence of solutions of \eqref{eq:hom}. In our case the whole manifold $M$ does not    admit a solution of \eqref{eq:hom}, but, as we show below, each
   indecomposable Riemannian block does admit a solution of \eqref{eq:hom}.
  
  Since $g$ has lorentzian signature, one of the metrics $g_s$,$s\geq 0$
  is a Lorentzian metric and all other metrics  are Riemannian.

  Suppose $g_0$ has  lorentzian signature. Since $\lambda_i$ is parallel, projection of $\lambda_i$ to each block is parallel. But indecomposable Riemannian blocks do not admit
  parallel  vector fields. Thus, $\lambda^i\in V_0$. On the
  other hand, we have shown that every parallel  vector
  field on $M$ is orthogonal to $\lambda_i$. Thus, $\lambda_i \in \ker
  g_0$. Then, $g_0$ is degenerate on $V_0$, which is a
  contradiction. Therefore, $g_0$ is Riemannian.

  Then, without loss of generality, we can assume that $(V_1,g_1)$ is
  Lorentzian indecomposable block and $\lambda_i\in V_1$.

   Let us now denote by   $M_s$ the integral submanifolds of the distribution generated by $V_s$; the restriction of the metric $g$ to $M_s$ will be denoted     by $\overset{(s)}g$. Our next goal is to construct a 1-form $\overset{s}{u}_i$  on each $M_s$ such that $\overset{s}{u}_{i,j}=\overset{(s)}{g}_{ij}. $
  For $s=0$ the existence of such   $\overset{0}{u}_{i}$ is trivial, since $\overset{(0)}g$ is flat. 
  
   We take  arbitrary $s>1$ and denote by $P^i_j$ the orthogonal
  projector of $T_p M$ to  $V_s$. Note that $P^i_j$ is  parallel. 
  For any vector field   $v^i$  on $M$ we consider the vector field  
  $u^i=P^j_i a_{jk} v^k$ on $M_s$.

  Its  covariant derivative   with respect to
  the index $k$ such that $\partial_k \in V_s$ is given by 
  $$
    u^i_{\ ,k}
    =(P^i_j a^j_l v^l)_{,k}=P^i_j a^j_{l,k} v^l+ P^i_j a^j_l
    v^l_{\ ,k}=\underbrace{P^i_j \lambda^j}_{=0}g_{lk} v^l+P^i_j \lambda_l 
    \delta^j_k v^l+ P^i_j a^j_l \underbrace{v^l_{\ ,k}}_{=0}
    =P^i_k \lambda_l v^l
  $$
implying that,  on $V_s$, we have 
  $u_{i,j}=(\lambda_l v^l) \overset{(s)}g_{ij}.$
  
  Since $V_1$ is  $g$-nondegenerate, there exists a
  vector field $v^i$ tangent to $M_1$  such that $\lambda_i v^i=1$. Then, the
  corresponding 1-form $u_i$ on $M_s$ satisfies the 
  property 
 $
    u_{i,j}=g_{i,j}
  $    as we what. This in particular implies that   every  block
  $V_s$ has dimension at least $3$, see Lemma~\ref{thm:Rv0}.

  Next we  consider  the  $(M_1, \overset{(1)}g)$. 
  We first  show that $(M_1, g_1)$ satisfies the
  assumptions of Corollary~\ref{thm:lorentz_dim}.

  We again  denote the orthogonal projector onto $V_1$ by $P$ and define the
  endomorphism  $a'$ on $V_1$   as a restriction of $a^i_j$ to  $V_1$:
  $a'=P\cdot a.$

  Since $P$ is  parallel  and $\lambda^i\in V_1$,
  $a'$ is a solution of~\eqref{basic} with respect to the  metric $\overset{(1)}g$. 
  Indeed, we take  indices $i,j,k$ such that $\partial_i, \partial_j, \partial_k\in TM_1$   and  calculate
  \begin{multline}
    a'_{ij,k}=(g_{ir}P^r_s a^s_j)_{,k}=g_{ir}P^r_s a^s_{j,k}=\\
    =g_{ir} P^r_s (\lambda^s g_{jk}+\lambda_j
    \delta^s_k)=g_{ir}\lambda^r g_{jk}+g_{ir}P^r_s \lambda_j
    \delta^s_k=\lambda_i g_{jk}+\lambda_j g_{ik}.
  \end{multline}

 We see that  $g_1$ admits at least three linearly independent
  solutions of the geodesic equivalence equations: $\const \cdot g_1$,
  $\lambda_i \lambda_j$ and $a'_{ij}$.  Thus, $g_1$ satisfies the
  conditions of the Theorem~\ref{thm:mg+Ba} and there exists the  unique
  constant $B(g_1)$ defined by the extended system~\eqref{eq:mg+Ba}.

  Since solution $(a',\lambda_i)$ with parallel  vector
  $\lambda_i$ and $\mu=0$ satisfies the extended
  system~\eqref{eq:mg+Ba} for $g_1$, from the second equation we
  obtain $B(g_1)=0$.
  
  Let us now show that $g_1$ does not admit a solution with $\mu
  \neq 0$.
   Assume  $(\tilde a,\tilde \lambda_i,\tilde \mu_i)$
 is  the solution of the extended system on $M_1$ with 
  $\tilde \mu\neq 0$. We can think 
  $\tilde \mu =1$.  Then, $\tilde\lambda_i$ is a 1-form  on
  $M_1$ such that
  $    \tilde \lambda_{i,j}=\tilde \mu g_{ij}=g_{ij}.$
  Let us now consider the sum 
  \begin{gather}
    \label{eq:cone-l}
    \xi_i=\overset{(0)}{u}_i+\tilde
    \lambda_i +\sum_{s=2}^\ell \overset{(s)}{u}_i, 
  \end{gather} where $\overset{(0)}{u_i}$ is a 1-form on  such
  that $\overset{(s)}{u}_{i,j}=\overset{s}g_{ij}$;  the existence of such 1-forms is proved above.  We evidently have 
 $\xi_{i,j}=g_{ij}$.
  We now consider  the symmetric $(0,2)$-tensor field
  $A_{ij}=\xi_i \xi_j$. 
  We have
  $$A_{ij,k}=\xi_i \xi_{j,k}+\xi_j \xi_{i,k}=\xi_i g_{jk}+ \xi_j g_{ik}.$$
  We see that $A_{ij}$ is a solution of \eqref{basic} with $\lambda_i= \xi_i$.
    Since $\xi_{i,j}=g_{ij}$,   the corresponding  $\mu$  equals  $1$. We obtain a 
  contradiction  with the assumption  that all solutions of the extended
  system  have $\mu=0$.

  Thus, we have shown that metric $g_1$ on the Lorentzian block
  does not admit solutions with
  $\mu\neq 0$.   Then, Lorentzian manifold $(M_1,g_1)$ satisfies the conditions of the
  Corollary~\ref{thm:lorentz_dim}. Thus, $\dim M_1 \geq 6$.
We therefore have (the number below $V_i$ corresponds to their dimensions)  
  $$
    T_p M=\underbrace{V_0}_{k_0}\oplus \underbrace{V_1}_{\geq 6}\oplus
    \underbrace{V_2}_{\geq 3}\oplus\dots\oplus \underbrace{V_\ell}_{\geq 3} \textrm{ implying } 
  $$
  $$
    \dim T_p M=k_0+\dim V_1+\dim V_2+\dots +\dim V_\ell
    \geq k_0 +6 +3(\ell-1)=3\ell+ k_0+3.
  $$  
  Thus, $1\leq \ell\leq \lfloor\frac{n-k_0}{3}\rfloor-1,$ where $k_0$
  is the dimension of the flat block $V_0$.
  Since the basis of  the   $1$-forms on $T_pM$ invariant w.r.t. the holonomy group is given by
  $k_0$ basis 1-forms  on the flat block $V_0$ and the 1-form $\lambda_i$ on $V_1$, we have $k=k_0+1$. Since the dimension of $V_1$ is at least 6, $0\leq
  k_0\leq n-3$.
 Therefore, $1\leq \ell\leq \lfloor\frac{n+1-k}{3}\rfloor-1,$ where
  $1\leq k\leq n-2$.
  Thus, we obtained the following list of values for the degree of
mobility of $g$:
  \begin{equation}\label{-11} D(g)=\dim Par(g)+1 =\tfrac{k(k+1)}{2}+\ell +1=\tfrac{k(k+1)}{2}+\ell' \textrm{
  where  } 2\leq \ell'\leq \lfloor\tfrac{n+1-k}{3}\rfloor.\end{equation}
   Lemma~\ref{thm:B=0_isotropic} is proved.

\section{Proof of the realization Theorem \ref{thm:realization}.}

\label{sec:realization_Bneq0}
In this section we construct an example of an $n$-dimensional
Lorentzian metric $g$  admitting a geodesically equivalent metric $\bar g$
that  is not affinely equivalent to $g$,  such that $D(g)= \frac{k(k+1)}{2}+\ell\geq
  2$, where $k$ and $\ell$ are as in Theorem \ref{thm:realization}.
  Essentially the same construction could be used for metrics of arbitrary 
  signature.  In  Section \ref{proof:apl} it will be explained 
  that for these metrics the number  $\dim proj (g) - \dim iso (g) $ equals $D(g)-1$ which implies that this example also shows that all the possible values of $\dim proj (g) - \dim iso (g) $ given by Theorem \ref{th:apl} can be achieved.

  Actually, we will  construct a $n+1$-dimensional cone manifold
  $(\wt M, \hat g)$  admitting  $\frac{k(k+1)}{2}+\ell$-dimensional space of 
   parallel symmetric $(0,2)$-tensor fields.  
  The metric $g$ 
  is then the restriction of $\hat g$ to the hypersurface $\{v =\const \}$, where $v$ is the function satisfying \eqref{eq:hom}.

  The manifold 
   $(\wt M, \hat g) $  will be the direct sum
   $$
   (\wt M, \hat g)= (\wt M_0, \hat g_0)+ ... + (\wt M_\ell, \hat g_\ell),  
   $$
  where  $(\wt M_0, \hat g_0)$ is the standard   $(\mathbb{R}^k, g_{euclidean})$ (in the case $k=0$ we think that $M_0$ is a point).
   Clearly, $(\wt M_0, \hat g_0)$ is a cone manifold  over the
  $(k-1)$-dimensional sphere with the standard metric.

  Since $\ell \leq\lfloor\frac{n-k+1}{3}\rfloor$, there exist numbers
  $k_1,\dots k_\ell$ such that $k_i\geq 3$ and $k_1+\dots+k_\ell=n-k+1$. For
  all $1\leq i\leq \ell$, as the manifold $(\wt M_i, \hat g_i)$ we take  the $k_i$-dimensional cone
  manifold over $(\R^{k_i-1},g_i)$,
  where $g_1$ is the flat metric of lorentzian signature and all $g_i,i\geq 2$ are
  euclidean (flat)  metrics. It is an easy exercise to prove   that  the manifolds $(\wt M_s, \hat g_s)$, $s\ge 1, $ do not admit a parallel symmetric $(0,2)$ tensor other than $\const \hat g_{s}$: one of the way to do this exercise is to calculate the curvature tensor $R^i_{\ jkm}$ and its  covariant 
  derivative $R^i_{\ jkm,s}$, which is possible since the metrics $g_i, i\ge 1$,   are  explicitly given by simple formulas, and to check that at each point $p\in \wt M_i$ the endomoprphisms $R^i_{\ j k m} X^k Y^m$ and  
  $R^i_{\ j k m , s } X^k Y^m Z^s$ generate the whole $so(T_p\wt M_i, \hat g_i)$.  
  Actually,  in the Riemannian case, i.e., for 
  $\wt M_s, \ s\ge 2$,  it follows  from 
   \cite[Theorem 4.1]{Cortes2009}.

  Evidently, 
  $\wt M=\wt M_0\times \wt M_1\times\dots\times \wt M_\ell$ has lorentzian signature.  It is a cone manifold by 
   Lemma~\ref{thm:glue}, so there exists a function $v$ satisfying $\eqref{eq:hom}$. As we explained in Section  \ref{sec:Bneq0}, the parallel symmetric $(0,2)$-tensor fields 
    on  $\wt M$ are in one-to-one correspondence with the solutions of \eqref{basic} for the restriction of the metric  to the ($n$-dimensional)
     hypersurface  $\{v=\const\}$, where $v$ is the function satisfying \eqref{eq:hom}, 
      so its dimension is the degree of mobility of $g$. 
    Combining  Theorem \ref{cov} and \ref{cov1}, we see that the 
     space of   parallel symmetric $(0,2)$-tensor fields on $(\wt M, \hat g)$   is precisely $\frac{k(k+1)}{2}+\ell$-dimensional. Theorem \ref{thm:realization} is proved. 
 
\section{Proof of  Theorem~\ref{th:apl}.} \label{proof:apl}

\subsection{Plan of the proof.} 
Our proof of  Theorem~\ref{th:apl} contains two main  parts: ``generic case'', which corresponds to the metrics with $D(g) \ge 3$ and $B\ne 0$, will be handled in Section \ref{gen}, 
 and ``special case'', when $B= 0$, will be handled in Section  \ref{spec}. 
In both cases, the upper bound  for $dim(proj)- dim(iso)$ follows from   Lemma~\ref{thm:proj/hom} in Section \ref{cod}.

We will always assume that our manifold $(M,g)$ is connected, simply-connected, has dimension $\ge 3$, and that there exists a metric $\bar g$ that is geodesically equivalent, but not affinely equivalent to $g$.

\subsection{Codimension of the space of homothety vector fields in the space of projective vector fields.}
 \label{cod} 
Recall that  a vector field $v$ is \emph{projective,} if its local
flow acts by projective transformations, i.e., takes unparametrized
geodesics to geodesics.
It is  well known  (see for example~\cite{s1,KM}) that the vector
field $v$ is projective if and only if the $(0,2)$-tensor field
$a^v$  given by the formula
\begin{equation}
  \label{apl}
  a^v:=  {\mathcal{L}}_v g - \tfrac{1}{n+1} \tr( g^{-1}
  {\mathcal{L}}_v g) \cdot g 
\end{equation}
satisfies the equation~\eqref{basic}. Here by $\mathcal{L}_v$ we
denote the Lie derivative along $v$.

\begin{lem}
  \label{thm:proj/hom}
  Let $(M^n, g)$ be a pseudo-Riemannian connected manifold  on $n$-dimensional manifold.
  
  Then,
  $\dim proj(g) - \dim
  hom (g)\leq D(g)-1,$ where $hom(g)$ denotes 
   be the space of homothety  (i.e., such that $\mathcal{L}_v = \const \cdot g$)
   vector fields. 
\end{lem}
\begin{proof}
  We denote by $Sol(g)$ the space of solutions of~\eqref{basic}
  corresponding to the metric $g$. Then by definition $\dim
  Sol(g)=D(g)$. Let $\widetilde{Sol}(g)$ be the quotient space
  $Sol(g)/\{const \cdot g\}$. Clearly, $\dim
  \widetilde{Sol}(g)=D(g)-1$.

  Let us consider the linear map $\phi:proj(g)\to \wtSol$, which maps
  each projective vector field to the corresponding equivalence class
  of the soluion $a^v$ given by the formula~\eqref{apl}.
  We show that $\ker \phi=hom(g)=\{v\mid \Lie_v=\const\cdot g\}$.

  Suppose $\phi(v)=0$. Then, there exists a constant $c$ such that 
  $a^v=c\cdot g$. Therefore,
  \begin{gather}
    \label{apl2}
    a^v:= {\mathcal{L}}_v g - \tfrac{1}{n+1} \tr( g^{-1}
    {\mathcal{L}}_v g) \cdot g=c\cdot g.
  \end{gather} 
  We multiply both sides by $g^{-1}$ and take the trace  to obtain 
  \begin{gather}
    \tr\left({g^{-1}\mathcal{L}}_v g - \tfrac{1}{n+1} \tr( g^{-1}
      {\mathcal{L}}_v g) \cdot \Id\right) =\tr( c\cdot \Id).
  \end{gather}
  Then,  $\tr\left({g^{-1}\mathcal{L}}_v g\right)$
  is constant. We substitute it in~\eqref{apl2} to obtain
  $\Lie_v g= \const \cdot g$  implying  $v$ is a homothety.

  Now, for a homothety vector field $v$ we have 
  $\Lie_v g=k\cdot g$, where $k$ is a certain constant. Then
  $$ a^v={\mathcal{L}}_v g - \tfrac{1}{n+1} \tr( g^{-1}
  {\mathcal{L}}_v g) \cdot g=k \cdot g - \tfrac{n\, k}{n+1} \cdot
  g=\tfrac{k}{n+1}\, g$$
  so that  $\phi(v)=0$.
  Thus, $\ker \phi=hom(g)$.
  
  Applying  the dimension theorem to the linear map $\phi:
  proj(g)\to image(\phi)\subset \wtSol(g)$, we  obtain
  $\dim proj(g)= \dim \ker \phi + \dim image(\phi)= 1 + \dim image(\phi).$
  Since  $\dim image(\phi)\leq \dim \wtSol$, 
  we obtain 
  $\dim proj(g) - \dim hom(g)\leq D(g)-1$ as we claimed. 
\end{proof}

\subsection{Generic case with $B\neq 0$.} \label{gen} 

Now we suppose that $D(g)\geq 3$ and that $B\neq 0$.

We consider  the linear map  
  $\phi$ from the proof of  Lemma~\ref{thm:proj/hom}. We need to show that 
 $\ker \phi=iso(g)$ and
  $image(\phi)=\wtSol(g)$.

  Since $\ker \phi=hom(g)$, in order to show that  $\ker \phi=iso(g)$  it is sufficient to show that 
  every
  homothety vector field is, in fact, a killing vector field.

  As we explained in Section \ref{standard1}, 
    $B$ is the global invariant of the metric, i.e. at  every
  point $P\in M$  the constant $B$  from the system \eqref{eq:mg+Ba}  is the same and  there exists,even locally,    only one  such constant 
  $B$. 

  Let $\bar g=F(t)^\ast (g)$ be the pullback of a metric $g$ with
  respect to the local flow of $v$. If $v$ is a homothety, then $\bar
  g=k\cdot g$ for some constant $k$. Then, as we explained in  Lemma~\ref{thm:two_cases}, 
  the constant $\bar B$ of
  the metric $\bar g$ is equal to $\bar b = \bar B(g)= \frac{1}{k}B$. But the metric $\bar g$ is isometric to the metric $g$, so it must have the same value of constant $B$. Thus, $k=1$, and the homothety vector field is in fact a Killing vector field as we claimed.

 Let us now show that  the image of $\phi$ is the
  whole space $\wtSol(g)$.  Choose the arbitrary $(a,\lambda,\mu)$
  from $Sol$ and consider the vector field $u^i=\lambda^i$. Then
  \begin{gather*}
    a^u_{ij}:= {\mathcal{L}}_u g_{ij} - \tfrac{1}{n+1} ( g^{pq}
    {\mathcal{L}}_u g_{pq}) \cdot g_{ij}=\\
    =\lambda_{i,j}+\lambda_{j,i}-\tfrac{1}{n+1} g^{pq}
    (\lambda_{p,q}+\lambda_{q,p}) \cdot g_{ij}=\\
    =2 \lambda_{i,j} - \tfrac{2}{n+1} g^{pq}
    \lambda_{p,q} g_{ij}=\\
    =2 (\mu g_{ij}+ B a_{ij})-\tfrac{2}{n+1} g^{pq} (\mu g_{pq} +B
    a_{pq}) g_{ij}= 2 B a_{ij} + \const\cdot g_{ij}.
  \end{gather*} 
  We see that, up to the an addition of $\const\cdot g_{ij}$,
  the solution $a^u$ is equal to  $2 B a_{ij}$. Now, we put
  $v^i=\tfrac{1}{2 B} \lambda^i$ and obtain
  $\phi(v)=\left[a\right]$, where with square brackets we
  denote the procedure of taking quotient in  $\wtSol$. Therefore, $image(\phi)=\wtSol$.
Finally,  
  $\dim proj - \dim iso = D(g) -1$
  as we claimed.

\subsection{Case $B=0$.} \label{spec} 

\subsubsection{ Assume there exists a solution $(a, \lambda, \mu)$ with $\mu\ne 0$. }
This case can be reduced to the case  $B\ne 0$ using the methods from  Section \ref{sec:zeroB}.  Indeed, in any connected simply-connected open subset of $M$ with compact closure we can find a metric $\bar g$ of the same signature  that is geodesically equivalent to $g$ and has $\bar B =B(\bar g)\ne 0$. Then, the restriction of the metric $g$ to
this subset has $ \dim prog(g)-\dim iso(g)$ as in Theorem \ref{th:apl}. Theorem  \ref{th:apl} follows then from the following

\begin{lem}
  \label{thm:noncompact-proj-iso}
  Let $(M,g)$ be a connected pseudo-Riemannian manifold and suppose
  $M=\cup_{s=1}^{\infty}M_s$, where $M_s$ are open connected subsets in
  $M$ and $M_s\subset M_{s+1}$. Denote by $g_s$ the restriction of $g$
  to $M_s$. Then, there exists $s$ such that for every $s'>s$
  $$\dim proj(g)=\dim proj( g_{s'})\quad \mbox{and}\quad \dim iso(g)=\dim iso(g_{s'})$$.
\end{lem}
The proof of this  Lemma is similar to   the proof of Lemma~\ref{thm:noncompact}, and will be left to the reader. The only essential property of the projective and isometry vector fields that should 
 be used in the proof is that     if two projective (respectively, isometric) vector fields coincide
  on an open subset of $M$, they coincide everywhere on $M$.

\subsubsection{ Special case: there is no  $(a, \lambda, \mu)$ with $\mu\ne 0$. }

Let us first proof 
 \begin{equation} \label{thm:B=0_isotropic_apl} D(g)-2\leq \dim proj(g)-\dim iso(g)\leq D(g)-1.\end{equation}

  The upper bound $  \dim proj(g)-\dim iso(g)\leq D(g)-1$ follows from 
  Lemma~\ref{thm:proj/hom}, since in view of Corollary \ref{cor3} the metric admits no homothety vector field.   
  
 In order to  prove $D(g)-2\leq \dim proj(g)-\dim iso(g)$, we   construct $D(g) - 2$  projective (in fact, affine) vector fields such that no nontrivial 
   linear combination of these vector fields is  a killing vector field; in this construction we will use  
  the description of the space $Sol(g)$ we have
  obtained in Section~\ref{sec:B=0_isotropic}.
We consider the decomposition of a tangent space $T_p M$
  \begin{gather}
    T_p M=V_0\oplus V_1\oplus\dots \oplus V_\ell
  \end{gather}
  where $V_0$ is maximal nondegenerate flat subspace of dimension
  $k_0= k-1$ and $V_s,1\leq s\leq \ell$ are indecomposable nondegenerate
  subspaces. 
In Section \ref{sec:B=0_isotropic} we have shown that   $D(g)=\dim Par (g) +1=\frac{k(k+1)}{2}+\ell+1$, see \eqref{-11}. We will  show that $g$ admits at least $D(g)-2= \frac{k(k+1)}{2}+\ell-1$
  affine vector fields such that no nontrivial 
   linear combination of these vector fields is  a killing vector field.

  We consider a  basis $\overset{(1)}{\tau}_i,\dots
  \overset{(k)}{\tau}_i$ of the $k$-dimensional space of all
  parallel  $1$-forms on $M$. 
  For each $\overset{(a)}{\tau}_i$ there exists function
  $\overset{(a)}{\tau}$ on $M$, such that
  $ \overset{(a)}{\tau}_i= \overset{(a)}{\tau}_{,i}$.
For every  $ a,b= 1,\dots,  k$,  we consider   the 1-form 
    $$ \overset{(ab)}u_i=  \overset{(b)}{\tau}
  \overset{(a)}{\tau}_i+\overset{(a)}{\tau} \overset{(b)}{\tau}_i.$$
  Every  $\overset{(ab)}u{}^i$ is an affine (and therefore projective) vector field   on $M$ (in the sence its local flow preserves the Levi-Civita connection), 
because of
   the Lie derivative of $g$ is parallel: indeed,  $$  \overset{(ab)}u_{i,j}=  \overset{(b)}{\tau}_j
  \overset{(a)}{\tau}_i+\overset{(a)}{\tau}_j
  \overset{(b)}{\tau}_i.$$
 Besides, we have $\ell-1$ additional  affine vector fields generated by the cone vector fields $\overset{s}{v}{}^i$  on   $M_s$ (such that ${\overset{s}{v}}_{i,j} = \overset{(s)}{g}{}_{ij}$, where  $
  s\geq2$. The vector fields $\overset{s}{v}{}^i$ are affine   since the Lie derivative of $g$ with respect  to  $\overset{s}{v}{}^i$ is $2\overset{(s)}{g}{}_{ij}$ and is a parallel $(0,2)$-tensor.
  
  No nontrivial linear combination of the vector fields $\overset{(ab)}u{}^i$, $a\le b$, and of $\overset{s}{v}{}^i$  is a killing vector field.  Indeed, the Lie derivatives of
  $g$  w.r.t. these vector fields are  linearly independent.

  Thus,
  $D(g)-1\geq\dim proj(g)-\dim iso(g)\geq \frac{k(k+1)}{2}+l-1\geq
  D(g)-2$ as we claimed.

In order to  explain that   \eqref{thm:B=0_isotropic_apl} implies the remaining part of 
 Theorem~\ref{th:apl},  we combine it with 
 Lemma~\ref{thm:B=0_isotropic} to obtain 
$$\frac{k(k+1)}{2} + \ell'-2=D(g)-2 \leq \dim proj(g)-
\dim iso(g) \leq D(g)-1=\frac{k(k+1)}{2} + \ell'-1$$ for certain $k\in
\{0,1,...,n-2, n\}$ and $2 \le \ell' \le \lfloor \tfrac{n-k+1}{3}\rfloor$.

Thus, $\dim proj(g)-\dim iso(g)=\frac{k(k+1)}{2} + \ell-1$, where
$\ell$ is either $\ell'$ or $(\ell'-1)$. Then $1 \le \ell \le \lfloor \tfrac{n-k+1}{3}\rfloor$ as we claimed in Theorem ~\ref{th:apl}. Theorem ~\ref{th:apl} is proved.

\subsubsection*{Acknowledgements}
We  thank  Ch. Boubel and P. Mounoud for useful discussions and  for checking the paper  and 
Deutsche Forschungsgemeinschaft for partial financial support.

\end{document}